\documentclass[notitlepage,reqno,11pt]{amsart}
\usepackage{latexsym,amssymb, epsfig, amsmath,amsfonts, subfigure,amsthm}

\usepackage{rotating}
\usepackage[toc,page]{appendix}
\usepackage{color}
\usepackage{multirow}
\usepackage{relsize}
\usepackage{microtype}
\usepackage{dsfont}
\usepackage{cases}

\usepackage[colorlinks, allcolors=blue]{hyperref}

\usepackage[margin=1in]{geometry}

\numberwithin{equation}{section}
\newtheorem{theorem}{Theorem}[section]
\newtheorem{lemma}{Lemma}[section]
\newtheorem{assumption}{Assumption}[section]

\newtheorem{coro}{Corollary}[section]
\newtheorem{prop}{Proposition}[section]
\newtheorem{remark}{Remark}[section]

\newtheorem{conjecture}{Conjecture}[section]

\newlength{\defbaselineskip}
\newcommand{\setlinespacing}[1]%
{\setlength{\baselineskip}{#1 \defbaselineskip}}

\def\E{\mathbb{E}}
\def\P{\mathbb{P}}

				\newcommand{\wt}{\widetilde}

				\newcommand{\R}  {\mathbb{R}}
				\newcommand{\N}  {\mathbb{N}}

				\newcommand{\non}{\nonumber}

				\newcommand{\baa}{\begin{eqnarray*}}
					\newcommand{\eaa}{\end{eqnarray*}}

				\newcommand{\ttl}{\Large
					Stochastic epidemic models with varying infectivity \\[5pt]  and waning immunity}
			
			\newcommand{\indic}[1]{\mathds{1}_{#1}}
			\renewcommand{\bar}[1]{\overline{#1}}

\begin{document}
				
				\title[]{\ttl}
				
				\author[Rapha\"el  Forien]{Rapha\"el Forien}
				\address{INRAE, Centre INRAE PACA, Domaine St-Paul - Site Agroparc
					84914 Avignon Cedex
					FRANCE}
				\email{raphael.forien@inrae.fr}

				\author[Guodong  Pang]{Guodong Pang}
				\address{Department of Computational Applied Mathematics and Operations Research, 
					George R. Brown School of Engineering and Computing,
					Rice University,
					Houston, TX 77005}
				\email{gdpang@rice.edu}
				
				\author[{\'E}tienne Pardoux]{{\'E}tienne Pardoux}
				\address{Aix Marseille Univ, CNRS, I2M, Marseille, France}
				\email{etienne.pardoux@univ.amu.fr}
				
				\author[Arsene Brice Zotsa--Ngoufack]{Arsene Brice Zotsa--Ngoufack}
				\address{Aix Marseille Univ, CNRS, I2M, Marseille, France}
				\email{arsene-brice.zotsa-ngoufack@univ-amu.fr}

				\date{\today}

				\begin{abstract} 
					We study an individual-based stochastic epidemic model in which infected individuals become susceptible again, after each infection. 
					In contrast to classical compartment models, after each infection, the infectivity is a random function of the time elapsed since infection.
					Similarly, recovered individuals become gradually susceptible after some time according to a random susceptibility function.
					
					We study the large population asymptotic behaviour of the model: we prove a functional law of large numbers (FLLN) and investigate the endemic equilibria of the limiting deterministic model. 
					The limit depends on the law of the susceptibility random functions but only on the mean infectivity functions. 
					The FLLN is proved by constructing a sequence of i.i.d. auxiliary processes and adapting the approach from the theory of propagation of chaos.
					The limit is a generalisation of a PDE model introduced by Kermack and McKendrick, and we show how this PDE model can be obtained as a special case of our FLLN limit.
					
					If $ R_0 $ is less than (or equal to) some threshold, the epidemic does not last forever and eventually disappears from the population, while if $ R_0 $ is larger than this threshold, the epidemic will not go extinct and there exists an endemic equilibrium.
					The value of this threshold turns out to be the harmonic mean of the susceptibility a long time after an infection.
				\end{abstract}
				

				\maketitle

				\allowdisplaybreaks
\section{Introduction}

Many infectious diseases are such that the immunity acquired after recovery from the illness is eventually lost, after a period whose length varies both with the individual and with the illness. 
The huge majority of the literature on mathematical epidemiology which considers models with loss of immunity assumes that this loss happens instantaneously (see for example \cite{britton2018stochastic,PandPardoux-2020}).
However, it is rather clear that, in reality, immunity wanes progressively over some period, which can vary from one individual to another (see for example \cite{singhal2020review,tuteja2007malaria}).
In their pioneering 1927 paper \cite{KMK}, Kermack and McKendrick introduced the first mathematical model of epidemic propagation in which the infectivity of individuals depends on the time elapsed since their infection (this time is often called the age of infection).
This model was deterministic and assumed that recovered individuals can no longer be infected.
In two subsequent papers, \cite{kermack_contributions_1932,kermack_contributions_1933}, Kermack and McKendrick introduced another deterministic model in which recovered individuals eventually lose their immunity.
They studied the conditions under which an infectious disease can become endemic, i.e., reach a stable equilibrium where some macroscopic fraction of the population is infected.
This kind of equilibrium for the deterministic system is called an endemic equilibrium, as opposed to the disease-free equilibrium, in which there are no infected individuals.

Our aim in the present paper is to revisit these questions with the use of a general stochastic model, following the recent developments carried out by the first three authors.
In \cite{forien2021epidemic}, the 1927 model of Kermack and McKendrick was obtained as the infinite population limit of an individual-based stochastic model in which the infectivity of individuals is a random function of their age of infection.
More precisely, it was shown that the fraction of susceptible individuals and the average infectivity (also called the force of infection) converge to a deterministic limit which solves a system of non-linear integral equations already introduced in \cite{KMK}. Note that an alternative to the integral equations model is a PDE model, see \cite{PP2023}.
The widespread ordinary differential equations (ODEs) compartmental SIR model is only a very special case of these equations, in which neither the infectivity nor the recovery rate depends upon the age of infection.
Such models are less realistic as they do not reproduce the dependence of the dynamics of the epidemic on its past \cite{sofonea2021memory,forien2021estimating}, and  thus neglect the inertia of the evolution of the epidemic. 
Note that ODE models are law of large numbers limit of Markov stochastic models, in which individuals move from one compartment to the next after an exponentially distributed time.

In the present paper, we introduce a very general model in which individuals are characterized by their infectivity and susceptibility, which are assumed to be given as random functions of their age of infection.
These random functions are assumed to be i.i.d. among the various individuals in the population, and a new independent pair of random functions is drawn at each new infection. See Figure~\ref{fig1} for a realization of the infectivity and susceptibility of an individual after his/her infection.

\begin{figure}[tbp] \label{fig1}
	\includegraphics[width=0.8\textwidth]{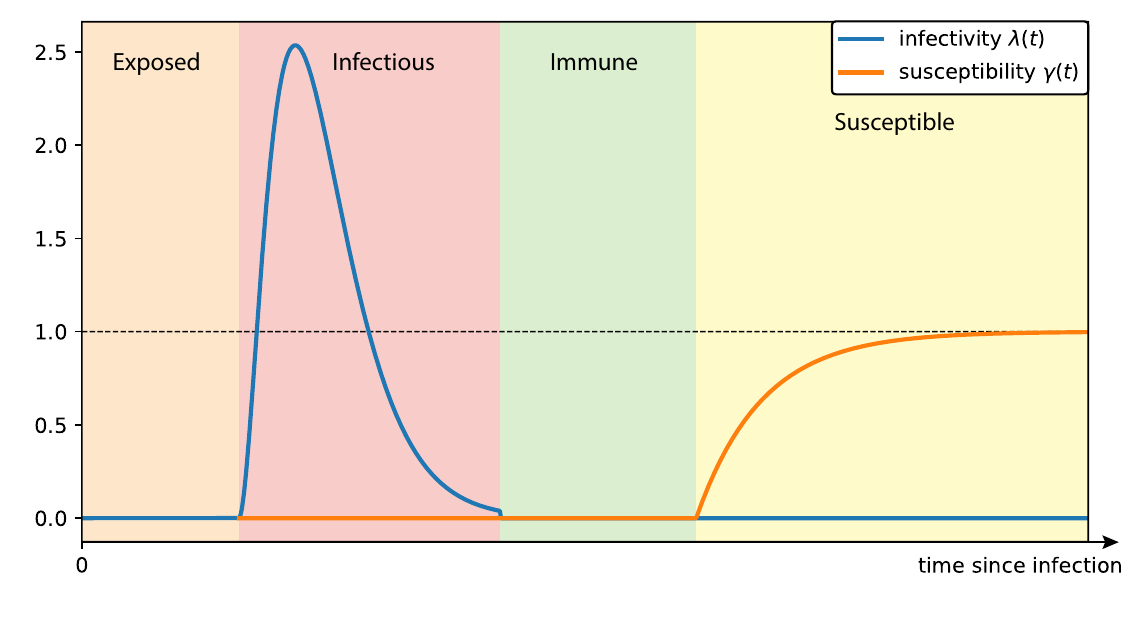}
	\caption{Illustration of a typical realization of the random infectivity and susceptibility functions of an individual from the time of infection to the time of recovery, and then to the time of losing immunity and becoming fully susceptible (or in general, partially susceptible).}
\end{figure}

Note however that, in contrast with the model in 
\cite{kermack_contributions_1932,kermack_contributions_1933}, we consider a closed population, in which there are no birth and no death (we do not exclude deaths due to infections, 
represented by individuals whose susceptibility remains equal to zero after the fatal infection).

We also assume (as in \cite{kermack_contributions_1932,kermack_contributions_1933}) that all pairs of individuals have contacts at the same rate. 
In other words, we assume a situation of homogeneous mixing. Of course, this is not quite satisfactory. 
One may wish to take into account the spatial distribution of the population, as well as the variety of social behaviors of the individuals. However, mathematical models involve necessarily a simplification of the complex reality. 
We believe that the results presented here constitute a significant progress over the classical models where all rates are constant and immunity is lost instantaneously.
In future works, we do intend to combine the complexity of the present work with that of inhomogeneous models, such as spatial models or models on graphons.

Besides the fact that we prefer to present our model in the form of a system of integral equations rather than a system of partial differential equations, our deterministic model is more general than the model of \cite{kermack_contributions_1932,kermack_contributions_1933}. The reason is the following. 
While the law of large numbers limit of the model with infection age dependent infectivity depends only on the mean of the random infectivity function \cite{forien2021epidemic}, the deterministic limit of the model with random susceptibility depends on the distribution of this random susceptibility function in a much more complicated way.
To see this, suppose that we wish to compute the average susceptibility of the individuals at some time $ t $.
This average can be obtained by summing the contributions of individuals that have not been infected on the interval $ [0,t) $, and of individuals that have been infected at some time $ s \in [0,t) $ and have not yet been reinfected again on the interval $ [s,t) $.
But the probability that such an individual has not been reinfected by time $ t $ depends on the full trajectory of their susceptibility, as well as on the trajectory of the force of infection, on the interval $ [s,t) $ (see \eqref{LLN_GF} below).
This explains why the deterministic limit is more complex than
that  in \cite{forien2021epidemic}, and shows that, in their 1932-33 model, Kermack and McKendrick implicitly assumed that the only random component of the susceptibility function is the time of recovery (see the discussion in Section~\ref{subsec:KMcK}).
As a consequence, we obtain a strict generalization of the model in \cite{kermack_contributions_1932,kermack_contributions_1933} as the deterministic limit of our stochastic model. 

Since in our model there is a flux of new susceptibles, due to waning of immunity, one expects that under certain conditions, there may be a stable endemic equilibrium.
We manage to study the existence, uniqueness and some stability properties of an endemic equilibrium (as well as the stability or instability of the disease--free equilibrium) in our  deterministic limit. 
We identify a threshold, which is the harmonic mean of the large time limit of the susceptibility (which is $1$ in Figure 1, but may be less than $1$ in our general model). 
If the basic reproduction number $R_0$ (defined below by \eqref{eqn-R0}) is smaller than this threshold, then the process converges to the disease-free equilibrium. We also show that under appropriate assumptions, if $R_0$ is larger than the threshold, and the model converges as $t\to\infty$ to some limit, then this limit is the unique endemic equilibrium, which is fully characterized. We conjecture that under appropriate assumptions, when $R_0$ is larger than the threshold, any solution of the deterministic limit starting with a non zero force of infection at time $t=0$ does converge to the endemic equilibrium. 

Let us comment on the method used to prove our law of large numbers result.  One key argument is based upon a coupling of  the processes which count the numbers of infections of the various individuals up to time $t$ with i.i.d. counting processes, where the renormalized force of infection is replaced by  its deterministic limit. This approach, for which the inspiration came from  \cite{chevallier2017mean}, is an application of ideas from the theory of propagation of chaos, see Sznitman \cite{sznitman1991topics}. Note both that we were unable to adapt the methodology of \cite{forien2021epidemic} to the setting of the present paper, and that the assumptions made here on  the random infectivity function are weaker than those made in \cite{forien2021epidemic}. Recently, the methodology of the present paper has been applied to  the model in \cite{forien2021epidemic}, thus resulting in the same result under weaker assumptions, see \cite{FPP-PUQR}. We also note that our approach  differs significantly from the techniques classically used for age-structured population models, as in \cite{oelschlager1990limit,jagers2000population,tran_large_2008,hamza2013age,hamza2016establishment,fan2020limit}. In these papers, the authors describe the model as a branching process (sometimes with interaction), which becomes an infinite dimensional Markov process if one adds all the age structure in the present state of the system, in which the lifespan, birth rate and death rate depend on the ages of all individuals in the population.
The state of such a model is then described by the empirical measure of the  ages of the individuals.
These works combine infinite-dimensional stochastic calculus and measure-valued Markov processes analysis in order to prove the convergence of the model.
In contrast, our stochastic models are non Markovian, and their deterministic limit does have a memory.

%

Using random infectivity and susceptibility functions allows us to build a very general model which is both versatile and tractable.
It captures the effect of a progressive loss of immunity, and this loss is allowed to be  different from one individual to another.
The integral equations that we obtain to describe the large population limit of our model are both compact and extremely general, since most epidemic models with homogeneous mixing and a closed population can be written in this form.
The effect of the variability of susceptibility on the endemic threshold has received very little attention in the literature, despite some profound implications which we outline in the present work.
The fact that the threshold depends on the harmonic mean of the susceptibility reached after an infection shows that the heterogeneity of immune responses in real populations should not be neglected in public health decisions.
Similarly, the variability of the immune response after vaccination (both in time and between individuals) should affect the efficacy of vaccination policies in non-trivial ways, although these questions are outside the scope of the present work.

We want to comment on the terminology which can be used for our model as a compartmental model. Recall that the compartments most classically used in epidemic models include S for Susceptible individuals, E for Exposed (those infected individuals who are not yet infectious), I for infectious and R for Recovered. We claim that our model can be classified as an SEIRS, SIRS, or SIS model. Indeed, referring to Figure~\ref{fig1}, we can consider that a given individual passes from the S to the E compartment when  he/she becomes infected, then into the I compartment when the attached infectivity first becomes positive, into the R compartment when the infectivity reaches $ 0 $ and remains null, and into the S compartment when the attached susceptibility becomes positive. However, without modifying the dynamics of the epidemic, we can merge the E and I compartments into a compartment I (for infected), where the infectivity need not be positive all the time. Similarly, we can merge the R and S compartments into the S compartment (or U, for uninfected, as we suggest below), whose members may have a susceptibility equal to zero. As a matter of fact, our model
is vey general and includes most of the existing homogeneous epidemic models without demography as particular examples. It is non Markov. All the rates are not only infection age dependent, but also random, i.e., different from one individual to another, which we believe reflects the reality of epidemics.

Models with gradual waning of immunity have been studied since Kermack and McKendrick by only a handful of authors, including Inaba, who in a series of works, see \cite{inaba2001kermack,inaba2004mathematical,inaba_endemic_2016,inaba_variable_2017}, has 
performed a careful mathematical study of the PDE model from \cite{kermack_contributions_1932,kermack_contributions_1933}, as well as Breda et al. \cite{breda2012formulation} who have considered an integral equation version of the same model. Other authors have pursued the study of the system of ODE/PDEs, see in particular Thieme and Yang \cite{ThieYang}, Barbarossa and R\"ost \cite{BarRos} and Carlsson et al. \cite{Carls}.  More recently, Khalifi and Britton
\cite{Khal-Brit},  compare the level of immunity in a model with gradual vs. sudden loss of immunity. We also mention the recent work  \cite{foutel2025optimal}, where a similar stochastic model of varying infectivity and waning immunity with vaccination is studied. 
We notice the difference from our modeling approach besides the vaccination aspect: the random susceptibility function and the random infectivity function are assumed mutually independent in each infection, contrary to what we assume here. In \cite{ngoufack2024functional} Zotsa-Ngoufack establishes the central limit theorem for the model described in the present paper. 

%


\subsection*{Organization of the paper} The rest of the paper is organized as follows. In Section \ref{sec-model}, we define the model in detail. In Section \ref{sec-FLLN}, we state the assumptions and the functional law of large numbers (FLLN) and discuss how the results reduce to already known results when we restrict ourselves to the classical SIS and SIRS models. The results on the endemic equilibrium are presented in Section \ref{sec-endemic-eq}. In Section \ref{sec-PDE},  we focus on the generalized SIRS model with a particular set of infectivity and susceptibility random functions and initial conditions, and show how the limit relates to the Kermack and McKendrick PDE model with the corresponding infection-age dependent infectivity and recovery-age dependent susceptibility. The proofs for the FLLN are given in Section \ref{sec-proofs-FLLN} and those for the endemic equilibrium in Section \ref{sec-proofs-endemic}.

\subsection*{Notation}
Throughout the paper, all the random variables and processes are defined on a common complete probability space $(\Omega, \mathcal{F},\P)$.  
We use $\xrightarrow[N\to+\infty]{\mathbb{P}}$ to denote convergence in probability as the parameter $N\to \infty$.
Let $\N$ denote the set of natural numbers and $\R^k (\R^k_+)$ the space of $k$-dimensional vectors with real (nonnegative) numbers, with $\R(\R_+)$ for $k=1$.  We use $\indic{\{\cdot\}}$ for the indicator function. Let $D=D(\R_+; \R)$ be the space of $\R$-valued c{\`a}dl{\`a}g functions defined on $\R_+$, with convergence in $D$ meaning convergence in the Skorohod $J_1$ topology (see, e.g., \cite[Chapter 3]{billingsley1999convergence}).  Also, we use $D^k$ to denote the $k$-fold product with the product $J_1$ topology. Let $C$ be the subset of $D$ consisting of continuous functions and $D_+$ the subset of $D$ of c{\`a}dl{\`a}g functions with values on $\R_+$

\bigskip 

\section{Model description}\label{sec-model}
\subsection{Definition of the model}
We consider a population of fixed size $N$. 
Let $(\lambda_0,\gamma_0)$ and $(\lambda,\gamma)$ be two random variables taking values in $D(\R_+, \R_+)^2$, and let $ \lbrace (\lambda_{k,0}, \gamma_{k,0}), 1 \leq k \leq N \rbrace $ be a family of i.i.d. copies of $ (\lambda_0, \gamma_0) $ and $ \lbrace (\lambda_{k,i}, \gamma_{k,i}), i \geq 1, 1 \leq k \leq N \rbrace $ be a family of i.i.d. copies of $ (\lambda, \gamma) $, independent of the previous family.
The function $\lambda_{k,0}$ (resp. $\gamma_{k,0}$) is the  infectivity (resp. susceptibility) of the $ k $-th individual between time $0$ and the time of his/her first (re)-infection, and $\lambda_{k,i}(t)$ denotes the infectivity of the $k$-th individual, at time $t$ after his/her $i$-th infection (where we count here only the infections after time $0$), and $\gamma_{k,i}(t)$ denotes the susceptibility of the $k$-th individual, at time $t$ after his/her $i$-th infection, given that this is his/her most recent infection.

We can think of the law of $(\lambda_0,\gamma_0)$ as being a mixture of the law for the initially infected individuals, who have been infected before time $0$ and for which $\lambda_0(0)\ge0$ and $\gamma_0(0)=0$, and the law for the initially susceptible individuals, for which $\lambda_0(0)=0$ and $\gamma_0(0)>0$ (possibly $\gamma_0(0)=1$).

We assume that $(\lambda_0,\gamma_0)$ and $(\lambda,\gamma)$ satisfy the following assumption.
\begin{assumption}\label{AS-lambda}
	We assume that:
	\begin{enumerate}
		\item $ 0 \leq \gamma_0(t)\leq1$ and $0\leq\gamma(t) \leq 1 $  almost surely and there exists a deterministic constant $ \lambda_* < \infty $ such that for all $ t \geq 0,\, 0\leq\lambda_0(t)\leq\lambda_*$ and $0\leq\lambda(t) \leq \lambda_* $ almost surely.
		\item Almost surely,
		\begin{multline}\label{eqq10}
			\sup\{t\geq0,\,\lambda_0(t)>0\}\leq\inf\{t\geq0,\,\gamma_0(t)>0\}\text{ and }\\ \sup\{t\geq0,\,\lambda(t)>0\}\leq\inf\{t\geq0,\,\gamma(t)>0\}.
		\end{multline}
	\end{enumerate}
\end{assumption}

Assumption~\ref{AS-lambda}-(ii) implies that, as long as an individual has not recovered, he/she cannot be reinfected.
Hence in each infected-immune-susceptible cycle, the infected and susceptible periods do not overlap.
Note that only (i) is necessary for the process to be well defined, and we discuss the consequences of removing part (ii) of the above assumption in Remark~\ref{remark:with_overlap} below. It would be rather natural to assume that $\gamma_0(t)$ and $\gamma(t)$ are a.s. non decreasing. We do not
make this restriction at this stage, since we do not need it. We shall make this assumption in Section~\ref{sec-endemic-eq} below.

For $ i \geq 0 $ and $ 1 \leq k \leq N $, we define 
\begin{align*}
	\eta_{k,0} := \sup \lbrace t \geq 0 : \lambda_{k,0}(t) > 0 \rbrace, \quad \text{ and }\quad
	\eta_{k,i} := \sup \lbrace t \geq 0 : \lambda_{k,i}(t) > 0 \rbrace.
\end{align*}
We will also use the notations
\begin{equation} \label{def:eta}
	\eta_0=\sup\{t>0,\ \lambda_0(t)>0\} \quad \text{ and } \quad \eta=\sup\{t>0,\ \lambda(t)>0\}.
\end{equation}

\begin{figure}[htb]
	\centering
	\includegraphics[width=\textwidth]{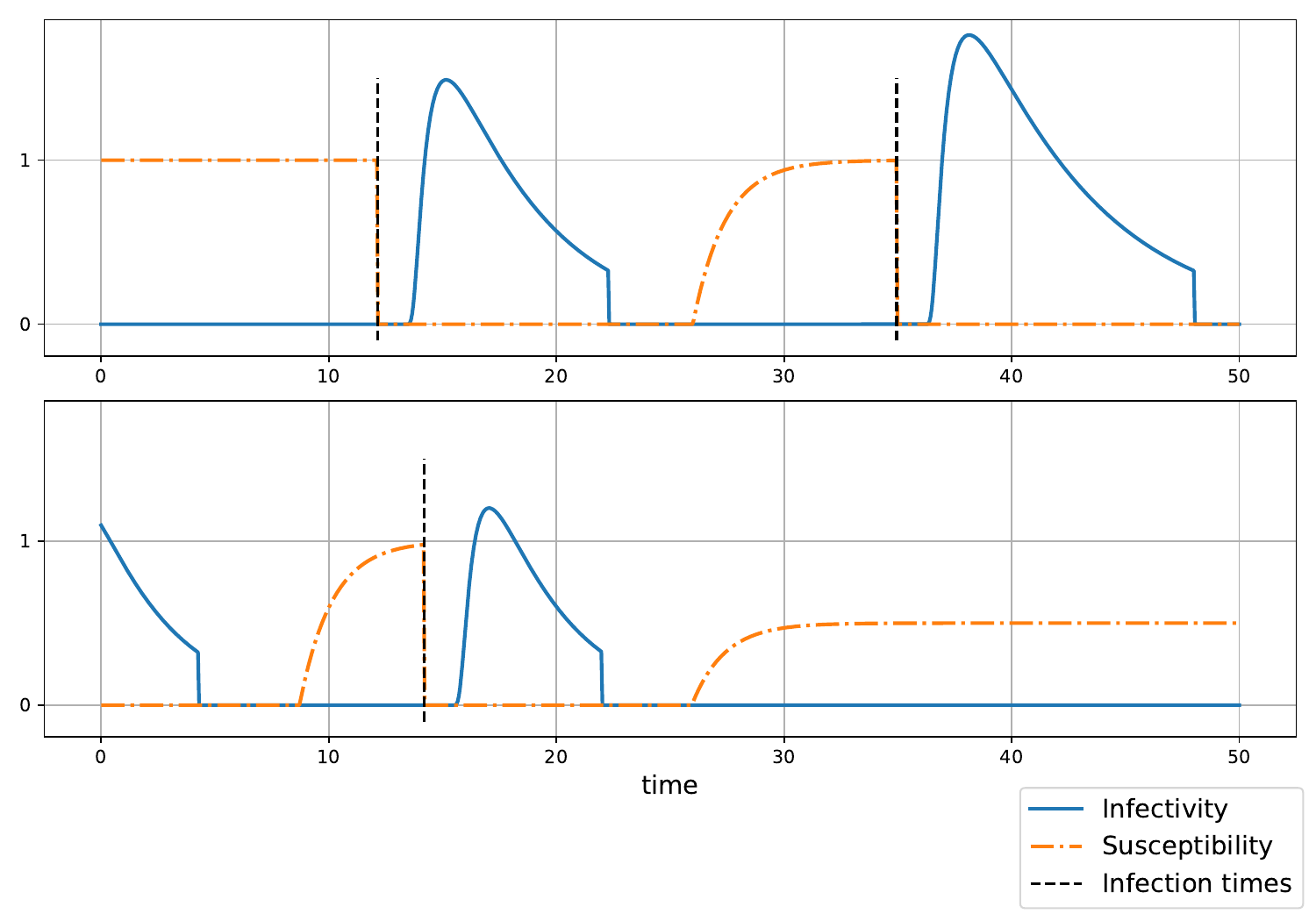}
	\caption{Illustration of the evolution of an individual's infectivity and susceptibility through time. Each graphic shows the dynamics of an individual's infectivity (blue) and susceptibility (orange). The top graphic corresponds to an individual which is initially susceptible, and the bottom one to an initially infectious individual. Note that, after being reinfected, the second individual remains partially immune even a long time after infection.} \label{fig:infection_cycles}
\end{figure}

\bigskip

Let $A^N_k(t)$ be the number of times that the individual $k$ has been (re)--infected on the time interval $(0,t]$.
The time elapsed since this individual's last infection (or since time $0$ if no such infection has occurred), is given by
\begin{align*}
	\varsigma^N_k(t) := t - \left(\sup\lbrace s \in (0,t] : A^N_k(s) = A^N_k(s^-) + 1 \rbrace\vee0\right),
\end{align*}
where we use the convention $\sup \emptyset = -\infty$.
With this notation, the current infectivity and susceptibility of the $k$-th individual are given by
\begin{align*}
	\lambda_{k,A^N_k(t)}(\varsigma^{N}_k(t)), && \text{ and } && \gamma_{k,A^N_k(t)}(\varsigma^{N}_k(t)).
\end{align*}
Figure~\ref{fig:infection_cycles} shows a realization of these processes for two individuals.
Let us now define $\overline{\mathfrak F}^N(t)$ and $ \overline{\mathfrak S}^N(t) $ as the average infectivity and susceptibility in the population, i.e.,

\begin{align*}
	\overline{\mathfrak{F}}^N(t) := \frac{1}{N} \sum_{k=1}^N \lambda_{k,A_k^N(t)}(\varsigma^N_k(t)), && \overline{\mathfrak{S}}^N(t):= \frac{1}{N} \sum_{k=1}^N \gamma_{k,A_k^N(t)}(\varsigma^N_k(t)). 
\end{align*}
The instantaneous rate at which the $ \ell $-th individual infects the $k$-th individual is
\begin{align*}
	\frac{1}{N} \lambda_{\ell,A^N_\ell(t)}(\varsigma^{N}_\ell(t)) \gamma_{k,A^N_k(t)}(\varsigma^{N}_k(t)) 
\end{align*}
where the $ \frac{1}{N} $ factor comes from the probability that the $ k $-th individual is chosen as the target of the infectious contact.
Summing over the index $\ell$, the instantaneous rate at which the $ k $-th individual is (re)infected is given by
\begin{align} \label{def_Upsilon}
	\Upsilon_k^N(t) := \gamma_{k,A^N_k(t)}(\varsigma^{N}_k(t))\, \overline{\mathfrak{F}}^N(t).
\end{align}

\bigskip

Let  now $ (Q_k, 1 \leq k \leq N) $ be an i.i.d. family of standard Poisson random measures (PRMs) on $ \R_+^2 $,  independent of the  sequence $(\lambda_{k,i},\gamma_{k,i})_{1\le k\le N, i\ge0}$.
The family of counting processes $(A^N_k(t), t \geq 0, 1 \leq k \leq N)$ is then defined as the solution of
\begin{align} \label{def_A_N_k}
	A^N_k(t) = \int_{[0,t] \times \R_+} \indic{u\leq\Upsilon^N_k(r^-)}Q_k(dr,du)\,.
\end{align}

Note that we construct $A^N_k$ by induction on the jumps times. Part $(i)$ of Assumption \ref{AS-lambda}  implies that the rate $\Upsilon_k^N(t)$ is bounded almost surely
by the constant $\lambda_\ast$. Consequently the jump times do not accumulate, and the above induction defines $A^N_k(t)$ for all $t\geq0$.

\subsection{Numbers of infected and uninfected individuals} \label{subsec:number_of_S_I}

For $ i \geq 1 $, $ \eta_{k,i} $ is the duration of the infected period of the $ k $-th individual following its $ i $-th infection, while $ \eta_{k,0} = 0 $ if the $ k $-th individual is initially susceptible and $ \eta_{k,0} > 0 $ is the remaining infected period of the $ k $-th individual if it is initially infected. 
We shall say that the $ k $-th individual is currently infected (resp. uninfected) if $ \varsigma^N_k(t) < \eta_{k,A^N_k(t)} $ (resp. $ \varsigma^N_k(t) \geq \eta_{k, A^N_k(t)} $).
Note that, with this definition, an individual need not be infectious when he/she is infected.
In the same way, an individual is called uninfected if he/she is \emph{no longer infected} or has never been infected, hence this group comprises both recovered and susceptible individuals.
This choice of two broadly defined compartments allows us to keep the notations tractable, but the equations obtained below can be generalized in a straightforward way to keep track of the number of individuals at different stages of their infection and susceptibility in more detail.

Let $ I^N(t) $ (resp. $ U^N(t) $) denote the number of infected (resp. uninfected) individuals in the population at time $ t \geq 0 $.
Then
\begin{align} \label{def_compartments}
	I^N(t) = \sum_{k=1}^{N} \indic{\varsigma^N_k(t) < \eta_{k,A^N_k(t)}}, && \text{ and } && U^N(t) = \sum_{k=1}^{N} \indic{\varsigma^N_k(t) \geq \eta_{k,A^N_k(t)}}.
\end{align}
Note that, quite obviously, $ U^N(t) + I^N(t) = N $ for all $ t \geq 0 $.
We then define
\begin{align*}
	\overline{I}(0) := \P(\eta_{0} > 0), && \overline{U}(0) := \P(\eta_{0} = 0) = 1-\overline{I}(0).
\end{align*}
Recall that $ \lbrace (\lambda_{k,0},\gamma_{k,0}), 1 \leq k \leq N \rbrace $, and hence $ (\eta_{k,0}, 1 \leq k \leq N) $, are independent and identically distributed.
Thus by the law of large numbers,
\begin{align*}
	\left( \frac{1}{N} I^N(0), \frac{1}{N} U^N(0) \right) = \left( \frac{1}{N} \sum_{k=1}^{N} \indic{\eta_{k,0} > 0}, \frac{1}{N} \sum_{k=1}^{N} \indic{\eta_{k,0} = 0} \right) \to \left( \overline{I}(0), \overline{U}(0) \right)
\end{align*}
almost surely as $ N \to \infty $.

\begin{remark} \label{rem-classical models}
	{\bf The classical compartmental models}
	
	Most compartmental epidemic models can be obtained as special cases of our model, as long as no population structure is assumed and birth and death are disregarded.
	Here are a few examples.
	\begin{enumerate}
		\item
		The SIS model considers that infected individuals instantly become infected with some deterministic infectivity $ \beta $, and remain infected for a random time $ \eta $, after which they become instantly fully susceptible again.
		Thus, this model is obtained by assuming that, 
		\begin{align} \label{def_lambda_gamma_SIS}
			\lambda(t) = \beta \indic{0 \leq t < \eta}, && \gamma(t) = \indic{t \geq \eta},
		\end{align}
		and $ \eta $ is a non-negative random variable, which follows an exponential distribution in the case of the Markov model.
		Let us also specify the distribution of $ (\lambda_{0}, \gamma_{0}) $.
		For this, fix $ \overline{I}(0) \in [0,1] $, set $ \overline{S}(0) = 1-\overline{I}(0) $ and let $ \chi$ be a random variable taking values in $ \lbrace S, I \rbrace $ such that $ \mathbb{P}(\chi = S) = \overline{S}(0) $ and $ \mathbb{P}(\chi = I) = \overline{I}(0) $.
		Then set
		\begin{align} \label{def_lambda0_gamma0_SIS}
			\lambda_{0}(t) = \begin{cases}
				0 & \text{ if } \chi = S, \\
				\beta \indic{0 \leq t < \eta_{0}} & \text{ if } \chi = I,
			\end{cases} && \gamma_{0}(t) = \begin{cases}
				1 & \text{ if } \chi = S, \\
				\indic{t \geq \eta_{0}} & \text{ if } \chi = I,
			\end{cases}
		\end{align}
		where $ \eta_{0} $ is a positive random variable (which follows an exponential distributions in the case of the Markov model).
		(Note that, with this definition, the model starts from a random initial condition, which is not always assumed in the literature, but does not greatly affect its behaviour.)
		
		\item The SIR model instead considers that, at the end of the infectious period, infected individuals recover from the disease, and can no longer be infected again.
		This model can be obtained from the above by proceeding as for the SIS model, but assuming instead that
		\begin{align} \label{def_gamma_SIR}
			\forall t \geq 0, \quad \gamma(t) = 0, && \text{ and } && \gamma_{0}(t) = \begin{cases}
				1 & \text{ if } \chi = S, \\
				0 & \text{ if } \chi = I.
			\end{cases}
		\end{align}
		Note that, in this case, if we keep a general distribution for the random functions $ \lambda $ and $ \lambda_{0} $, this model reduces to the one studied in \cite{forien2021epidemic}.

		\item
		The SIRS model assumes that, at the end of their infectious period, individuals stay immune to the disease for a random time $ \theta $, after which they become fully susceptible again.
		This model is obtained by assuming that,
		\begin{align} \label{def_lambda_gamma_SIRS}
			\lambda(t) = \beta \indic{0 \leq t < \eta}, && \gamma(t) = \indic{t \geq \eta + \theta},
		\end{align}
		where $ (\eta, \theta),$ is a pair of independent random variables taking values in $ \R_+^2 $ (which are distributed as pairs of independent exponential variables in the Markov models).
		The generalization of the definition of the distribution of $ (\lambda_{0}, \gamma_{0}) $ to this case is straightforward.
		
		Note that, in the last two cases, the quantity $ U^N(t) $ defined in \eqref{def_compartments} counts both susceptible and removed individuals.
		The actual number of susceptible and removed individuals in the SIRS model is given by
		\begin{align} \label{def_RN}
			S^N(t) := \sum_{k=1}^{N} \indic{\eta_{k,A^N_k(t)} + \theta_{k,A^N_k(t)} \leq \varsigma^N_k(t)}, &&  R^N(t) := \sum_{k=1}^{N} \indic{\eta_{k,A^N_k(t)} \leq \varsigma^N_k(t) < \eta_{k,A^N_k(t)} + \theta_{k,A^N_k(t)}},
		\end{align}
		where 
		\[\theta_{k,i}=\inf\{\theta>0,\gamma_{k,i}(\eta_{k,i}+\theta)>0\},\]
		following \eqref{def_lambda_gamma_SIRS}.

		In the SIR model, the above expression remains exact provided we set $ \theta_{k,i} = +\infty $ for $ i \geq 1 $, $ \theta_{k,0} = 0 $ if $ \chi_k = S $ and $ \theta_{k,0} = +\infty $ if $ \chi_k = I $.
	\end{enumerate}
	It is also common to assume that infected individuals do not become infectious right after being infected, but first become \textit{exposed} (i.e. infected but not yet infectious) before becoming infectious.
	This results in an additional compartment E, which can also be included in the above examples without difficulty. 
\end{remark}

\begin{remark}
	In \cite{chevallier2017mean}, Chevallier studied a model formulated as a system of age-dependent random Hawkes processes.
	This model considers a system of $ N $ neurons which fire at a rate depending both on the times of the previous firings of other neurons and on the time elapsed since their last firing (called the age process).
	The author proves in \cite{chevallier2017mean} a propagation of chaos result for the empirical measure of the point processes corresponding to the firing times of the neurons, and for the empirical measure of the age processes of the neurons, as $ N $ tends to infinity.
	Although neither our model or that of Chevallier can be formulated as a special case of the other, the two are closely related (if firing is understood as an analogous of being infected).
	The main difference between the two frameworks is in the assumptions on the randomness in the interaction between individuals after each firing/infection.
	For instance, our model would be closer to that of Chevallier if, instead of choosing a different infectivity function after each infection for each individual, we chose a different infectivity function for each \emph{directed pair} of individuals at the \emph{beginning}, and kept the same infectivity function for this pair of individual after each infection.
	Thus, the law of large numbers limit that we prove below can be seen as an extension of Chevallier's result, and indeed some steps of the proof are adapted from \cite{chevallier2017mean}.
	See also Theorem~\ref{LLN-P} in Section~\ref{sec-proofs-FLLN} whose formulation is closer to the propagation of chaos result of \cite{chevallier2017mean}.
	In \cite{chevallier_fluctuations_2017}, the same author also proves a central limit theorem for the empirical measures mentioned above, something which we do not do here, but will be the subject of future work.
\end{remark}

\bigskip 

\section{Functional law of large numbers} \label{sec-FLLN}

In this section we present the FLLN for the scaled processes $\big(\overline{\mathfrak{S}}^N, \overline{\mathfrak{F}}^N, \overline{U}^N, \overline{I}^N\big) $ where $\overline{U}^N:=N^{-1}U^N$ and $\overline{I}^N:=N^{-1}I^N$. 

Let 
\begin{align*}
	&\overline{\lambda}_0(t):=\E\big[\lambda_{0}(t)\Big|\eta_{0}>0\big],\,\quad \overline{\lambda}(t):=\E\left[\lambda(t)\right],\\
	&F^c_0(t):=\P\left(\eta_0>t\big|\eta_{0}>0\right),\quad F^c(t):=\P\left(\eta>t\right),
\end{align*}
and recall that $ \overline{I}(0)=\P\left(\eta_{0}>0\right) =\P(\chi_1=1)$.
We shall denote below by  $\mu$ the law of  $\gamma$.

To describe the limits of the FLLN, we introduce the following two-dimensional integral equation. 
Observe that the solution $(x,y)$ of this system depends on the laws of $\lambda_0$ and $\lambda$ only through the expectations $ \overline{\lambda} $ and $ \overline{\lambda}_0 $, but depends in a much more complex way on the law of $ \gamma $ and $ \gamma_0 $.

We consider the following system of integral equations for which we look for a solution $(x,y)\in D^2_+:$
\begin{equation}\label{LLN_xy}
	\left\lbrace
	\begin{aligned}
		x(t) &= \E\left[ \gamma_0(t) \exp\left(-\int_0^t \gamma_0(r) y(r) dr 	\right)\right]  \\
		&\qquad \qquad  + \int_0^t \E\left[ \gamma(t-s) \exp\left(- \int_s^t \gamma(r-s) y(r) dr \right)\right]x(s) y(s)ds, \\
		y(t) &=  \bar{I}(0)\overline {\lambda}_0(t) + \int_0^t \overline{\lambda}(t-s) x(s) y(s)ds.
	\end{aligned}
	\right.
\end{equation}
In \eqref{LLN_xy} the only random elements are $\gamma_0$ and $\gamma$. 

Our first result establishes existence and uniqueness of the solution $(x,y)$ of \eqref{LLN_xy}.  The proof is given in Section \ref{sec-proofs-FLLN}. 

\begin{theorem}\label{ExistUniq}
	Under Assumption~\ref{AS-lambda}, the set of equations \eqref{LLN_xy} has a unique solution $(\overline{\mathfrak{S}},\overline{\mathfrak{F}})\in D^2(\R_+,\R_+)$. The solution belongs
	to $C^2(\R_+)$ if $\gamma_0$ has bounded variation and the map $t\mapsto(\E\left[\gamma_0(t)\right],\overline{\lambda}_0(t))$ is continuous.
\end{theorem}

We have the following convergence result which we prove in Section~\ref{sec-proofs-FLLN}. 
\begin{theorem}\label{thm-FLLN}
	Under Assumption~\ref{AS-lambda},
	\begin{equation} \label{eqn-mfk-SF-conv}
		(\overline{\mathfrak{S}}^N,\overline{\mathfrak{F}}^N)\xrightarrow[N\to+\infty]{\mathbb{P}}(\overline{\mathfrak{S}},\overline{\mathfrak{F}})\quad\text{ in }\quad D^2
	\end{equation}
	where  $(\overline{\mathfrak{S}},\overline{\mathfrak{F}})$ is the unique solution of the system of equations \eqref{LLN_xy}.
	
	Given the solution   $(\overline{\mathfrak{S}},\overline{\mathfrak{F}})$, 
	\[(\overline{U}^N,\overline{I}^N)\xrightarrow[N\to+\infty]{\mathbb{P}} (\overline{U},\overline{I})\quad\text{ in }\quad D^2\]
	where $(\overline{U},\overline{I})$ is given by 
	\begin{align}
		\overline{U}(t)&=\E\left[ \indic{t\geq\eta_0} \exp\left(-\int_0^t \gamma_0(r) \overline{\mathfrak{F}}(r) dr \right)\right]  \non
		\\ &\qquad+ \int_0^t \E\left[ \indic{t-s\geq\eta} \exp\left(- \int_s^t \gamma(r-s) \overline{\mathfrak{F}}(r)dr \right)\right]\overline{\mathfrak S}(s) \overline{\mathfrak F}(s)ds\,, \label{eqn-barS}\\
		\overline{I}(t)&=\overline{I}(0)F^c_0(t)+\int_{0}^{t}F^c(t-s)\overline{\mathfrak{S}}(s)\overline{\mathfrak{F}}(s)ds\,. \label{eqn-barI}
	\end{align} 
	
\end{theorem}
We rewrite \eqref{LLN_xy} as 
\begin{equation}\label{LLN_GF}
	\left\lbrace
	\begin{aligned}
		\overline{\mathfrak{S}}(t) &= \E \left[\gamma_{0}(t)\exp \left( - \int_{0}^{t} \gamma_0(r) \overline{\mathfrak F}(r) dr \right)\right] \\
		&\qquad \qquad + \int_{0}^{t} \E \left[\gamma(t-s) \exp \left( - \int_{s}^{t} \gamma(r-s) \overline{\mathfrak F}(r) dr \right) \right] \overline{\mathfrak S}(s) \overline{\mathfrak F}(s) ds, \\
		\overline{\mathfrak{F}}(t) &= \overline{I}(0)\overline{\lambda}_0(t)+\int_{0}^{t}\overline{\lambda}(t-s)\overline{\mathfrak S}(s)\overline{\mathfrak F}(s)ds\,.
	\end{aligned}
	\right.
\end{equation}
We prove the following Lemma in Section~\ref{sec-proofs-FLLN}.
\begin{lemma}\label{rem-1}
	If the pair $(x,y)$ is a solution to the set of equations \eqref{LLN_xy}, then for all $ t \geq 0 $,
	\begin{equation}\label{eqG1}
		\E \left[ \exp \left( - \int_{0}^{t} \gamma_0(r) y(r) dr \right)\right] + \int_{0}^{t} \E \left[ \exp \left( - \int_{s}^{t} \gamma(r-s) y(r) dr \right) \right] x(s) y(s) ds = 1.
	\end{equation}
	
\end{lemma}
\begin{remark} \label{remark:conservation_of_mass}
	Since $ \overline{U}^N(t) + \overline{I}^N(t) = 1 $ for all $ t \geq 0 $ and $ N \geq 1 $, it follows from the above convergence that $ \overline{U}(t) + \overline{I}(t) = 1 $ as well.
	Let us check that this follows also from the set of equations \eqref{LLN_GF} satisfied by $(\overline{\mathfrak S},\overline{\mathfrak F})$.
	
	First we note that, by \eqref{eqq10} in Assumption~\ref{AS-lambda}, $ \gamma(t) = 0 $ for all $ t \in [0, \eta) $ and $ \gamma_{0}(t) = 0 $ for $ t \in (0,\eta_0) $, hence
	\begin{align*}
		F_0^c(t) &= \E \left[ \indic{t < \eta_0} \exp\left( - \int_{0}^{t} \gamma_{0}(r) \overline{\mathfrak F}(r) dr \right) \right], \\ F^c(t-s) &= \E \left[ \indic{t-s < \eta} \exp\left( \int_{s}^{t} \gamma(r-s) \overline{\mathfrak F}(r) dr \right) \right].
	\end{align*}
	Hence, summing \eqref{eqn-barS} and \eqref{eqn-barI}, we obtain
	\begin{multline} \label{I+S}
		\overline{U}(t) + \overline{I}(t) = \E \left[ \exp \left( - \int_{0}^{t} \gamma_0(r) \overline{\mathfrak F}(r) dr \right)\right] \\+ \int_{0}^{t} \E \left[ \exp \left( - \int_{s}^{t} \gamma(r-s) \overline{\mathfrak F}(r) dr \right) \right] \overline{\mathfrak S}(s) \overline{\mathfrak F}(s) ds.
	\end{multline}
	Therefore, as $(\overline{\mathfrak S},\overline{\mathfrak F})$ is a solution of the set of equations \eqref{LLN_xy}, by Lemma~\ref{rem-1} we conclude that $ \overline{U}(t) + \overline{I}(t) = 1 $ for all $ t \geq 0 $.
	The identity \eqref{eqG1} can be seen as stating the conservation of the population size.
\end{remark}

\begin{remark} \label{remark:with_overlap}
	Without condition $\eqref{eqq10}$ of Assumption~\ref{AS-lambda}, the limit obtained in the functional law of large numbers satisfies a different set of equations.
	More precisely, the second equation in \eqref{LLN_GF} is replaced by 
	\begin{multline}\label{eqG2s-G-w}
		\overline{\mathfrak{F}}(t)=\E \left[\lambda_{0}(t)\exp \left( - \int_{0}^{t} \gamma_0(r) \overline{\mathfrak F}(r) dr \right)\right] \\
		+ \int_{0}^{t} \E \left[\lambda(t-s) \exp \left( - \int_{s}^{t} \gamma(r-s) \overline{\mathfrak F}(r) dr \right) \right] \overline{\mathfrak S}(s) \overline{\mathfrak F}(s) ds, 
	\end{multline}
	and \eqref{eqn-barI} by 
	\begin{multline}\label{eqG2s-F-w}
		\overline{I}(t)=\E\left[ \indic{\eta_0>t} \exp\left(-\int_0^t \gamma_0(r) \overline{\mathfrak{F}}(s) dr \right)\right] 
		\\ + \int_0^t \E\left[ \indic{\eta>t-s} \exp\left(- \int_s^t \gamma(r-s) \overline{\mathfrak{F}}(r)dr \right)\right]\overline{\mathfrak S}(s) \overline{\mathfrak F}(s)ds\,.
	\end{multline}
	\end{remark}
	
	\begin{remark}
		We can check that, in the special cases mentioned in Remark~\ref{rem-classical models}, the limiting system of equations obtained in Theorem~\ref{thm-FLLN} coincides with the corresponding models in the literature.
		\begin{enumerate}
			\item In the case of the SIS model, we note that, by \eqref{def_lambda_gamma_SIS} and \eqref{def_lambda0_gamma0_SIS},
			\begin{align*}
				\overline{\lambda}(t) = \beta \P(\eta > t), && \overline{\lambda}_0(t) = \beta \P(\eta_0 > t \,|\, \eta_0> 0).
			\end{align*}
			It thus follows from \eqref{LLN_GF} and \eqref{eqn-barI} that $ \overline{\mathfrak F}(t) = \beta \overline{I}(t) $ for all $ t \geq 0 $.
			Moreover, comparing \eqref{LLN_GF} and \eqref{eqn-barS} and using \eqref{def_lambda_gamma_SIS} and \eqref{def_lambda0_gamma0_SIS}, we see that $ \overline{\mathfrak{S}}(t) = \overline{U}(t)=\overline{S}(t) $.
			Combining this with the fact that $ \overline{I}(t) + \overline{U}(t) = 1 $, we obtain
			\begin{equation} \label{eqn-barI-SIS}
				\bar{I}(t) = \bar{I}(0) F^c_0(t) + \beta \int_0^t F^c(t-s) (1-\bar{I}(s)) \bar{I}(s) ds, 
			\end{equation}
			as stated in Theorem~2.3 of \cite{PandPardoux-2020}.
			
			\item In the case of the SIR model, from the definition of $ \gamma $ and $ \gamma_0 $ in \eqref{def_gamma_SIR} and \eqref{LLN_GF}, combined with the fact that $ \overline{\mathfrak F}(t) = \beta \overline{I}(t),$ we see that
			\begin{align*}
				\overline{S}(t)=\overline{\mathfrak S}(t) = \overline{S}(0) \exp\left( - \beta\int_{0}^{t} \overline{I}(r) dr \right) \quad \text{ and } \quad  \overline{U}(t)=\overline{ S}(t)+\overline{R}(t).
			\end{align*}
			This yields the statement of Theorem~2.1 in \cite{PandPardoux-2020}.
			
			\item In the case of the SIRS model, we note that, $ \overline{\mathfrak F}(t) = \beta \overline{I}(t) $, $ \overline{\mathfrak{S}}(t) = \overline{S}(t)$, in view of \eqref{eqn-barS}, $ \overline{U}(t) = \overline{S}(t) + \overline{R}(t) $, where
			\begin{multline*}
				\overline{S}(t) = \E\left[ \indic{\eta_0 + \theta_0 \leq t} \exp \left( -\beta \int_0^t \gamma_{0}(r) \overline{I}(r) dr \right) \right] \\ + \beta\int_{0}^{t} \E\left[ \indic{\eta + \theta \leq t-s} \exp \left( -\beta \int_{s}^{t} \gamma(r-s) \overline{I}(r) dr \right) \right] \overline{S}(s) \overline{I}(s) ds,
			\end{multline*}
			and
			\begin{multline*}
				\overline{R}(t) = \E \left[ \indic{\eta_0 \leq t < \eta_0 + \theta_0} \exp \left( - \beta\int_{0}^{t} \gamma_{0}(r) \overline{I}(r) dr \right) \right] \\ + \beta\int_{0}^{t} \E\left[ \indic{\eta \leq t-s < \eta + \theta} \exp\left( -\beta \int_{s}^{t} \gamma(r-s) \overline{I}(r) dr \right) \right] \overline{S}(s) \overline{I}(s) ds.
			\end{multline*}
			In fact, $ \overline{S} $ and $ \overline{R} $ are the limit of $ \frac{1}{N} S^N $ and $ \frac{1}{N} R^N $, respectively, where $ S^N $ and $ R^N $ are defined in \eqref{def_RN} (recall that, for simplicity, we assumed that no individual is initially in the R compartment). 
			But, from the definition of $ \gamma $ in \eqref{def_lambda_gamma_SIRS}, $\gamma(t) = 0$ for all $ t \in [\eta, \eta + \theta) $ and $ \gamma_{0}(t) = 0 $ for all $ t \in [\eta_0, \eta_0+\theta_0) $, hence
			\begin{multline*}
				\overline{S}(t) = \overline{S}(0) \exp \left( -\beta \int_{0}^{t} \overline{I}(r) dr \right) + \overline{I}(0) \E \left[ \indic{\eta_0 + \theta_0 \leq t} \exp \left( - \beta\int_{\eta_0+\theta_0}^{t} \overline{I}(r) dr \right) \right] \\ + \beta\int_{0}^{t} \E \left[ \indic{\eta + \theta \leq t-s} \exp \left( - \beta\int_{s + \eta + \theta}^{t} \overline{I}(r) dr \right) \right] \overline{I}(s) \overline{S}(s) ds,
			\end{multline*}
			and
			\begin{align*}
				\overline{R}(t) = \overline{I}(0) \P(\eta_0 \leq t < \eta_0 + \theta_0 \,| \, \eta_0 > 0) + \beta\int_{0}^{t} \P(\eta \leq t-s < \eta + \theta) \overline{I}(s) \overline{S}(s) ds.
			\end{align*}
			As $\overline{S}(t) + \overline{I}(t) + \overline{R}(t) = \overline{U}(t) + \overline{I}(t) = 1,$
			this yields the result stated in Theorem~3.3 in \cite{PandPardoux-2020}.
		\end{enumerate}
	\end{remark}
	
	\bigskip
	
	\section{The endemic equilibrium} \label{sec-endemic-eq}
	
	When the disease persists in the population, we say that the disease becomes endemic, and if it reaches an equilibrium, it is called the endemic equilibrium.
	This corresponds to a balance between the number of new infections and new recoveries.
	In this section, we study the conditions that lead to an endemic equilibrium in the FLLN limit.
	We first recall the endemic equilibrium behavior of  the classical SIS model discussed in Remark \ref{rem-classical models}. 
	Given the infectivity rate $\beta$, the basic reproduction number is given by $R_0=\beta\E[\eta]$.
	It is well known (see, e.g., \cite[Section 4.3]{PandPardoux-2020}, and also \cite{britton2018stochastic} for a discussion of the Markov model) that if $R_0 \le 1$, $\bar{I}(t)\to 0$ as $ t \to\infty $,  and if $R_0>1$ and $\bar{I}(0)>0,\,\bar{I}(t)\to 1-R_0^{-1}$ as $ t \to \infty $.  
	This can be easily obtained from the expression of $\bar{I}(t)$ in \eqref{eqn-barI-SIS}. In words, in the case $R_0\le 1$,  the disease-free steady state is globally asymptotically stable, and in the case $R_0>1$, 
	the disease-free steady state is unstable and there is exactly one endemic steady state, which is globally asymptotically stable. 
	Recall that in this model, $\overline{\lambda}(t) =\beta F^c(t)$. Thus, since $\int_0^\infty F^c(t)dt = \E[\eta]$, we have 
	\begin{equation}\label{eqn-R0}
		R_0 = \int_0^\infty \overline{\lambda}(t)dt.
	\end{equation}  
	In fact, the expression of $R_0$ in \eqref{eqn-R0} is the definition of the basic reproduction number in the Kermack and McKendrick model with an average infectivity function $\overline{\lambda}(t)$, because it represents the average number of individuals infected  by an infectious individual in a fully susceptible population.

	We make the following assumptions on the random susceptibility and infectivity functions in order to study the equilibria of our model. 
	
	\begin{assumption}\label{hyp-ga-1}
		\begin{enumerate}
			\item The random variable $ \eta $ defined in \eqref{def:eta} is integrable. 
			\item The random function $t \mapsto \gamma(t)$ is almost surely non-decreasing. 
			\item The pair $(\lambda_0,\gamma_0)$ is distributed as follows.
			Let $\xi\geq0$ be a random variable such that $\xi\leq\eta$ almost surely. 
			Let $\chi$ be a Bernoulli random variable with $\P(\chi=1)=\overline{I}(0)$, independent of $(\lambda,\gamma,\eta,\xi)$. Then
			\begin{equation*}
				\lambda_{0}(t)=\begin{cases}
					0&\mbox{ if }\chi=0,\\\lambda(t+\xi)&\mbox{ if }\chi=1.
				\end{cases}
				\quad \text{ and }\quad
				\gamma_{0}(t)=\begin{cases}
					1&\mbox{ if }\chi=0,\\\gamma(t+\xi)&\mbox{ if }\chi=1.
				\end{cases}
			\end{equation*} 
			The random variable $\xi$ represents the age of infection at time zero of the initially infectious individuals. 
		\end{enumerate}
	\end{assumption}
	
	Note that, by the definition of $ \gamma_0 $ and by (ii) $ t \mapsto \gamma_0(t) $ is also non-decreasing almost surely.
	In addition, since $ \lambda(t) \leq \lambda_* $ almost surely, by (i) and the dominated convergence theorem,
	\begin{equation} \label{lim_lambda_bar}
		\lim_{t \to \infty} \overline{\lambda}(t) = 0.
	\end{equation}

	We define
	\begin{align*}
		\gamma_\ast:=\sup_{t\geq0}\gamma(t)=\lim_{t\to+\infty}\gamma(t), \quad \text{ and } \quad \gamma_{0,*} := \lim_{t\to+\infty} \gamma_{0}(t).
	\end{align*}
	In the classical SIS model, the susceptibility functions are given by $\gamma(t) = \indic{t \geq \eta} $ and $ \gamma_{0}(t) = \indic{t \geq \eta_0}$ so that $\gamma_*=1$ a.s. However, in general, after being infected and recovered, individuals lose immunity gradually and do not necessarily reach ``full'' susceptibility (being equal to 1). Thus,  $\gamma_\ast$ may take any value in $[0,1]$ and is a priori random. 
	
	It turns out that the classification of the endemic equilibria depends on the law of $\gamma_*$, more specifically on whether $R_0$ is smaller than or larger than $\E[1/\gamma_*]$. Note that this expectation $\E[1/\gamma_*]$ may be infinite in general (for example if $ \P(\gamma_* = 0) > 0 $). 
	We first prove the following result under the condition that  $R_0 < \E[1/\gamma_*]$. 
	
	\begin{theorem}\label{thm-eqlm1}
		Under Assumptions~\ref{AS-lambda} and \ref{hyp-ga-1}, 
		if $R_0 < \E \left[ 1/\gamma_* \right],$ 
		there exists $ \overline{\mathfrak S}_* $ such that 
		\begin{align*}
			(\overline{\mathfrak S}(t),\overline{\mathfrak F}(t)) \to ( \overline{\mathfrak S}_*,0) \quad as \quad t \to \infty,
		\end{align*}
		where 
		\begin{multline} \label{eqn-mfkS*-1}
			\overline{\mathfrak S}_\ast=\lim_{t\to+\infty}\overline{\mathfrak S}(t) = \E\left[ \gamma_{0,*} \exp\left(-\int_0^{+\infty} \gamma_0(r) \overline{\mathfrak{F}}(r) dr \right)\right] 
			\\ + \int_0^{+\infty} \E\left[ \gamma_\ast \exp\left(- \int_s^{+\infty} \gamma(r-s) \overline{\mathfrak{F}}(r) dr \right)\right]\overline{\mathfrak{S}}(s) \overline{\mathfrak{F}}(s)ds.
		\end{multline}
		As a consequence, 
		$(\bar{U}(t), \bar{I}(t)) \to (1, 0)$ as $t\to\infty$. 
		In the special case $\gamma_*=1$ a.s., we also have $\overline{\mathfrak S}_*=1$. 
	\end{theorem}
	\begin{remark}
		Without the assumption on the monotonicity of the function $\gamma$, when $R_0<\E\left[\left(\sup_t\gamma(t)\right)^{-1}\right]$ the same proof as that of Theorem~\ref{thm-eqlm1} shows that as $t\to+\infty,\,\overline{\mathfrak F}(t)\to0$ and $\overline{I}(t)\to0.$  
	\end{remark}
	
	Note that in this theorem,  we do not assume $\E \left[1/\gamma_* \right]<+\infty$. However, we do assume that $R_0<\infty$, that is, $\overline\lambda(t)$ is integrable.  
	\begin{remark}
		In \cite[Proposition~$8.9$]{inaba_variable_2017} one can find a similar result for a model with demography.
	\end{remark}

	The case  $R_0 \geq \E[1/\gamma_*]$ is more complex. We make the following additional assumptions. 
	
	\begin{assumption}\label{hyp-ga}
		There exists  a non-negative random variable $t_\ast$ such that $\E\left[t_\ast\right]<+\infty$ and for $t\geq t_{\ast},\,\gamma(t)\geq\frac{\gamma_\ast}{2}$ a.s.  
	\end{assumption}

	\begin{assumption}\label{ASS2}
		$\gamma_\ast$ is deterministic, positive and for any $\delta \in (0,1)$,  there exists a deterministic $t_\delta>0$  such that 
		\begin{equation}
			\gamma_0(t_\delta)\wedge\gamma(t_\delta)\geq(1-\delta)\gamma_\ast\text{ almost surely.}
		\end{equation}
	\end{assumption}
	\begin{assumption}\label{ASS3}
		There exists a positive decreasing function $h$ such that $h(0)=1$ and for all $s,t\in\R_+,\;\overline{\lambda}(s+t)\geq h(s)\overline{\lambda}(t)$. The same holds for $\overline{\lambda}_0$. In addition,  $\overline{\lambda}_0$ is continuous and $\overline{\lambda}$ is of bounded total variation.
	\end{assumption}

	\begin{theorem}\label{conv1-f}
		\begin{enumerate}
			\item\label{conv1-f-i} Suppose that Assumptions \ref{AS-lambda}, \ref{hyp-ga-1} and~\ref{hyp-ga} hold and  $R_0> \E \left[ \frac{1}{\gamma_*} \right]$.  
			If there exists $(\overline{\mathfrak{S}}_\ast,\overline{\mathfrak{F}}_\ast)$ such that $(\overline{\mathfrak{S}}(t),\overline{\mathfrak{F}}(t))\xrightarrow[t\to+\infty]{}(\overline{\mathfrak{S}}_\ast,\overline{\mathfrak{F}}_\ast)$,
			then either $\overline{\mathfrak F}_\ast=0$,  or else
			\[\overline{\mathfrak{S}}_\ast=\frac{1}{R_0}\] and $\overline{\mathfrak{F}}_\ast$ is the unique positive solution of the equation 
			\begin{equation}\label{ibar1} 
				\int_{0}^{+\infty}\E\left[\exp\left(-\int_{0}^s \gamma\left(\frac{r}{\overline{\mathfrak{F}}_\ast}\right) dr \right)\right] ds=R_0.
			\end{equation}  	
			In the second case, $ (\overline{I}(t), \overline{U}(t)) \to (\overline{I}_*, \overline{U}_*) $ as $ t \to \infty $, where $ \overline{U}_* = 1-\overline{I}_* $ and
			\begin{equation} \label{eqn-I*-indicator-1} 
				\bar{I}_* =  \frac{\E\left[\eta \right] \overline{\mathfrak F}_*}{R_0}\,. 
			\end{equation}
			If $R_0= \E \left[ \frac{1}{\gamma_*} \right]$, the same statement holds but \eqref{ibar1} does not admit any positive solution, so, necessarily, $\overline{\mathfrak F}_\ast=0$ and thus $\bar{I}_*=0$.
			\item\label{thm-nostable} Assume in addition that Assumption~\ref{ASS2} and Assumption~\ref{ASS3} hold and that $\overline{\mathfrak{F}}(0)>0$.  Then there exists $c>0$ such that for all $t>0,\,\overline{\mathfrak{F}}(t)\geq c.$ In particular $ \overline{\mathfrak F}(t) $ cannot tend to zero as $ t \to \infty $.
		\end{enumerate}
		
	\end{theorem}
	
	\begin{remark}
		Assumption~\ref{hyp-ga} controls the time required for an individual's susceptibility to increase beyond some level,
		Assumption~\ref{ASS2} ensures that after some time the average susceptibility returns above $\frac{1}{R_0}$ if there are not too many re-infections and Assumption~\ref{ASS3} ensures that the force of infection does not decrease too rapidly.
	\end{remark}
	
	\begin{coro} \label{cor:endemic_eq_special_case}
		Suppose that Assumptions \ref{AS-lambda}, \ref{hyp-ga-1} and~\ref{hyp-ga} hold and that  $R_0> \E \left[ \frac{1}{\gamma_*} \right]$.  
		Assume that  $\gamma(t) = \gamma_\ast \indic{t \geq \zeta}$, where $\zeta$ is a random variable satisfying $\zeta\geq\eta$ almost surely and $\E[\zeta]<+\infty$ and $ \gamma_* $ is a random variable taking values in $ (0,1] $. 
		Then, 
		\begin{equation} \label{eqn-mfkF*-indicator} 
			\overline{\mathfrak F}_\ast=\frac{R_0-\E\left[\frac{1}{\gamma_*}\right]}{\E\left[\zeta\right]}, 
		\end{equation}
		is the solution of equation \eqref{ibar1} and
		\begin{equation} \label{eqn-I*-indicator} 
			\bar{I}_* =  \frac{\E[\eta] \overline{\mathfrak F}_*}{R_0} =  \frac{\E[\eta]}{\E[\zeta]}\left(1 - R_0^{-1} \E\left[\frac{1}{\gamma_*}\right]\right)
		\end{equation}
		where $R_0$ is given by \eqref{eqn-R0}. 
	\end{coro}
	
	\begin{proof}
		In this case, equation \eqref{ibar1} becomes 
		\begin{align*}
			\int_{0}^{+\infty}\E\left[\exp\left(-\gamma_\ast\int_{0}^s    \indic{ \frac{r}{\overline{\mathfrak{F}}_\ast}  \ge \zeta}  dr \right)\right] ds=R_0.
		\end{align*}
		By a change of variables and Fubini's theorem, the left hand side is equal to
		\begin{multline*}
			\overline{\mathfrak{F}}_\ast\int_{0}^{+\infty}\E\left[\exp\left(-\overline{\mathfrak{F}}_\ast\gamma_\ast\int_{0}^s    \indic{r  \ge \zeta}  dr \right)\right] ds\\
			\begin{aligned}
				&= \overline{\mathfrak{F}}_\ast\int_{0}^{+\infty} \E \left[   \indic{ \zeta >s} + \indic{\zeta \leq s} \exp\left(-\overline{\mathfrak{F}}_\ast\gamma_\ast (s- \zeta)   \right)\right]   ds\\
				&= \overline{\mathfrak F}_*\E\left[\zeta\right] +  \overline{\mathfrak{F}}_\ast \E\left[ \int_{\zeta}^{+\infty} \exp\left(-\overline{\mathfrak{F}}_\ast\gamma_\ast (s-\zeta)  \right) ds \right]  \\
				&=\overline{\mathfrak F}_*\E\left[\zeta\right]+\E\left[\frac{1}{\gamma_*}\right].
			\end{aligned}
		\end{multline*}
		This gives the expression in \eqref{eqn-mfkF*-indicator}. 
	\end{proof} 
	
	\begin{remark} \label{remark:endemic_eq_classical_models}
		For the classical models discussed in Remark~\ref{rem-classical models}, the above allows us to recover previously known results.
		
		\begin{enumerate}
			\item In the SIS model, we have $\gamma_*=1$ and $ \eta = \zeta $ almost surely, so we obtain that $ \overline{I}(t) \to 0 $ as $ t \to \infty $ if $ R_0 \leq 1 $, and that, if $ R_0 > 1 $, the only other possible limit for $ \overline{I}(t) $ is given by
			\begin{align*}
				\overline{I}_* = \frac{\E[\eta]}{\E[\zeta]} \left(1 - R_0^{-1} \E\left[\frac{1}{\gamma_*}\right]\right) = 1-\frac{1}{R_0}.
			\end{align*}
			
			
			\item For the SIRS model, we have $ \gamma_* = 1 $ and $ \zeta = \eta + \theta $.
			As a result, applying Corollary~\ref{cor:endemic_eq_special_case}, we see that, when $ R_0 > 1 $, the only possible positive limit for $ \overline{I}(t) $ is
			\begin{align*}
				\overline{I}_* = \frac{\E[\eta]}{\E[\eta] + \E[\theta]} \left( 1-\frac{1}{R_0} \right).
			\end{align*}
			In the same way, we can also deduce that $ \overline{R}_* := \lim_{t\to+\infty} \overline{R}(t) $ is given by
			\begin{align*}
				\overline{R}_* = 1- \overline{\mathfrak S}_* - \overline{I}_* =  \frac{\E[\theta]}{\E[\eta] + \E[\theta]} \left( 1- \frac{1}{R_0} \right).
			\end{align*}
			We refer to Proposition 4.2 of \cite{PandPardoux-2020} for a previous derivation of this equilibrium in the particular case where $\lambda(t)=\beta\indic{t<\eta}$ and $\gamma(t)=\indic{t\geq\eta+\theta}$ with $\eta$ and $\theta$ independent, the laws of $\eta$ and $\theta$ being arbitrary.
		\end{enumerate}
		
		Our results in \eqref{eqn-mfkF*-indicator}--\eqref{eqn-I*-indicator} thus extend those for classical compartmental models. Note that $\gamma_*$ is random and takes values in $[0,1]$, indicating potentially partial susceptibility after recovery.
	\end{remark}

	For the SIS and SIRS models discussed in Remark~\ref{remark:endemic_eq_classical_models}, it is known that, if $ R_0 > 1 $ and $ \overline{I}(0) > 0 $, then $ \overline{I}(t) $ does indeed converge to the endemic equilibrium $ \overline{I}_* $ as $ t \to \infty $.
	We are not yet able to prove such a result for our general model.

	We make the following conjecture on the convergence to the equilibrium in the case $ R_0 > \E \left[1/\gamma_* \right] $.

	\begin{conjecture}\label{conj:convergence}
		Under Assumptions~\ref{AS-lambda}, \ref{hyp-ga-1}--\ref{ASS3}, if $ R_0 > \E \left[ \frac{1}{\gamma_*} \right] $ and $ \overline{\mathfrak F}(0) > 0 $, then 
		\begin{align*}
			(\overline{\mathfrak S}(t),\overline{\mathfrak F}(t)) \to (\overline{\mathfrak S}_*,\overline{\mathfrak F}_*) \quad as \quad t\to\infty,
		\end{align*}
		where $ \overline{\mathfrak S}_* = 1/R_0 $ and $ \overline{\mathfrak F}_* $ is the unique positive solution of \eqref{ibar1}.
	\end{conjecture}
	
	This conjecture is equivalent to another. 
	Indeed, by applying the Arzel{\`a}-Ascoli Theorem to the set of functions $\lbrace(\tau_{j}\overline{\mathfrak S},\tau_{j}\overline{\mathfrak F}),\, j \in \N \rbrace$ where 
	$\tau_{j}x(t):=x(t+t_j)$ and $(t_j)_j\subset\R_+,\,t_j\to+\infty$ as $j\to+\infty$, there exists a subsequence of pairs $(\tau_{j}\overline{\mathfrak S},\tau_{j}\overline{\mathfrak F})$ denoted again $(\tau_{j}\overline{\mathfrak S},\tau_{j}\overline{\mathfrak F})$ such that $(x_j,y_j):=(\tau_{j}\overline{\mathfrak S},\tau_{j}\overline{\mathfrak F})\to(x,y)$ uniformly on compact sets as $j\to+\infty$.
	Note that the pair $(x_j(t),y_j(t))$ satisfies the following system of equations: for $t\geq-t_j$,
	\begin{equation} \label{eq:xj_yj}
		\left\lbrace
		\begin{aligned}
			x_j(t) &= \E\left[ \gamma_0(t+t_j) \exp\left(-\int_0^{t+t_j} \gamma_0(r) y_j(r) dr \right)\right]  \\
			&\qquad \qquad  + \int_{-t_j}^t \E\left[ \gamma(t-s) \exp\left(- \int_s^t \gamma(r-s) y_j(r) dr \right)\right]x_j(s) y_j(s)ds,\\
			y_j(t) &=  \bar{I}(0)\overline {\lambda}_0(t+t_j) + \int_{-t_j}^t \overline{\lambda}(t-s) x_j(s) y_j(s)ds.
		\end{aligned}
		\right.
	\end{equation}
	As a result, as the first terms of the right hand side of \eqref{eq:xj_yj} tend to zero when $j\to+\infty$, and $(x_j(t),y_j(t))\to(x(t),y(t))$ for all $t\in\R$, we deduce by the dominated convergence theorem that the pair $(x,y)$ satisfies the following set of equations,
	\begin{equation} \label{eq:y}
		\left\lbrace
		\begin{aligned}
			&y(t)=\int_{-\infty}^t \overline{\lambda}(t-s)x(s)y(s)ds, \\
			&\int_{-\infty}^{t}\E\left[\exp\left(-\int_{s}^{t}\gamma(r-s)y(r)dr\right)\right]x(s)y(s)ds=1.
		\end{aligned}
		\right.
	\end{equation}
	We can remark that the constant pair $(\frac{1}{R_0},\overline{\mathfrak F}_\ast)$ where $\overline{\mathfrak F}_\ast$ is the unique solution of \eqref{ibar1} is a solution of \eqref{eq:y}. 
	Hence if this solution is unique, all converging subsequences of $(\tau_{j}\overline{\mathfrak S},\tau_{j}\overline{\mathfrak F})$ have the same limit, from which we can easily conclude the convergence of $ (\overline{\mathfrak F}(t), \overline{\mathfrak S}(t)) $ as $ t \to \infty $.
	
	Thus, Conjecture~\ref{conj:convergence} is equivalent to the following.
	
	\begin{conjecture}
		Under Assumptions~\ref{hyp-ga-1}--\ref{ASS3}, if $ R_0 > \E \left[ \frac{1}{\gamma_*} \right] $ and $ \overline{\mathfrak F}(0) > 0 $, the set of equations \eqref{eq:y} has a unique positive and bounded solution on $\R$. 
	\end{conjecture}

	\bigskip
	
	\section{Relating to the Kermack and McKendrick PDEs for the SIRS model  } \label{sec-PDE}

	\subsection{The FLLN limits with a special set of susceptibility/infectivity functions and initial conditions}
	
	We consider the special family of susceptibility and infectivity functions:
	\[
	\lambda(t)=\wt{\lambda}(t)\indic{t<\eta}\, \quad\text{ and }\quad \gamma(t)=\wt{\gamma} (t-\eta)\indic{t>\eta}
	\]
	where $\wt{\lambda}(t)$ and $\wt{\gamma}(t)$ are deterministic functions such that $\wt{\lambda}(t)\leq\lambda_*$ and $\wt{\gamma}(t)\leq1$ for all $ t \geq 0 $ for some $ \lambda_* > 0 $ and the infected period $\eta$ is a random variable with cumulative distribution function $ F $, so that $ (\lambda, \gamma) $ satisfy Assumption~\ref{AS-lambda}. 
	Then we have $\overline{\lambda}(t) =\wt{\lambda}(t)F^c(t)$.  
	In addition, let $f$ be the density function of $F$ and $\mu_F(t) := \frac{f(t)}{F^c(t)}$ be the hazard rate function of $F$. 
	Then one can also write
	\begin{equation}\label{N2}
		F^c(t)=\exp\left(-\int_{0}^{t}\mu_F(s)ds\right),
	\end{equation} 
	
	We assume in the present Section~\ref{sec-PDE} that $\mu_F $ is bounded on $ \R_+ $.
	Hence $F^c(t)>0$ for all $ t \geq 0 $.
	
	We assume that the infectivity function of an initially infected individual is the same as that of a newly infected individual but shifted by the time elapsed since infection before time zero, that is, 
	\[
	\lambda_0(t)=\wt{\lambda}(t+\xi)\indic{t<\eta_0}
	\]
	where $\eta_0$ is the duration of the remaining infectious period after time zero, whose distribution depends on $\xi$, the age of infection at time $0$.
	To specify the distribution of $\xi$, we assume that there exists a function $\bar{I}(0,\tau)$ such that $\bar{I}(0) = \int_0^\infty \bar{I}(0,\tau) d \tau$, that is, $\bar{I}(0,\tau)$ is the  density of the distribution of $\bar{I}(0)$ over the ages of infection.  Then we can specify the distribution of $\xi$ as 
	\begin{equation}\label{eqp3}
		\mathbb{P}(\xi>x)=\frac{1}{\overline{I}(0)}\int_{x}^{+\infty}\overline{I}(0,r)dr.\end{equation}
	And the conditional distribution of $\eta_0$ given $\xi$ is given by 
	\begin{equation}
		\label{lc}
		\mathbb{P}(\eta_0>t\big|\xi) = \frac{F^c(\xi+t)}{F^c(\xi)}=\exp\left(-\int_{\xi}^{\xi+t}\mu_F(r)dr\right). \end{equation}
	It is then clear that
	\begin{align} \label{eqm-barI0-lambda0}
		\overline{\lambda}_0(t) = \E\left[\wt{\lambda}(t+\xi)\indic{t<\eta_0}\right] =  \frac{1}{\overline{I}(0) } \int_{0}^{+\infty}\wt\lambda(t+\tau)\overline{I}(0,\tau)\exp\left(-\int_{\tau}^{t+\tau}\mu_F(s)ds\right)d\tau\,. 
	\end{align}
	
	Next, to specify the initial susceptibility $\gamma_0(t)$ associated to each individual, we consider three groups of individuals at time zero, susceptible, infected, and recovered, whose respective proportions are $\overline{S}(0),\,\overline{I}(0)$ and $\overline{R}(0)$ with $\overline{S}(0)+\overline{I}(0)+\overline{R}(0)=1$.
	Note that we have combined the fully susceptible and initially recovered individuals as one group in the model description in Section \ref{sec-model}. Thus, the model discussed in this section is in fact a generalized SIRS model. 
	Moreover, we shall in this section consider that individuals in the $ R $ compartment have recovered from the disease but are not necessarily immune. In accordance with \cite{inaba2001kermack} (see Section~\ref{subsec:KMcK} below), once an individual recovers from the infection, after some potential immune period, the immunity is gradually lost, the individual may become infected again, and the process is repeated at each new infection with a new realization of the random infectivity and susceptibility functions. This means that, the susceptible individuals are those who have never been infected.
	For the initially fully susceptible individuals, their susceptibility $\gamma_0$ equals $1$. 
	For the initially infected individuals, their susceptibility starts to become positive after time $\eta_0$, which is the end of the infected period $\eta_0$, i.e. $ \gamma_{0}(t) = \indic{t \geq \eta_0} \wt{\gamma}(t-\eta_0) $.
	For the initially recovered individuals, their susceptibility depends on the time elapsed since the last recovery, which we denote by $\vartheta$, so that, for these individuals, $ \gamma_{0}(t) = \wt{\gamma}(t + \vartheta) $. 
	Let $\chi$ be a random variable indicating the group of an individual at time $t=0$, with the following law
	\[\mathcal{L}(\chi)=\overline{S}(0)\delta_{S}+\overline{I}(0)\delta_{I}+\overline{R}(0)\delta_{R}\,. \]
	Then the initial susceptibility $\gamma_0(t)$ of an individual can be written as 
	\[
	\gamma_{0}(t)=\indic{\chi=S}+\wt{\gamma}(t-\eta_0)\indic{t\geq\eta_0,\chi=I}+\wt{\gamma}(t+\vartheta)\indic{\chi=R}\,. 
	\]
	To specify the distribution of $\vartheta$, we assume that there exists a function $\bar{R}(0,\theta)$ such that $\bar{R}(0) = \int_0^\infty \bar{R}(0,\theta)d \theta$, that is,  $\bar{R}(0,\theta)$ is the density of the distribution of $\bar{R}(0)$ over the age of recovery. Then we can specify the distribution of $\vartheta$ by 
	\begin{equation}\label{eqp2}
		\mathbb{P}(\vartheta>x)=\frac{1}{\overline{R}(0)}\int_{x}^{+\infty}\overline{R}(0,r)dr.\end{equation}
	Thus, we have 
	\begin{multline*}
		\E \left[\gamma_{0}(t)\exp \left( - \int_{0}^{t} \gamma_0(r) \overline{\mathfrak F}(r) dr \right)\right] =  \overline{S}(0)  \exp \left( - \int_{0}^{t}  \overline{\mathfrak F}(r) dr \right) \\
		+ \int_{0}^{t} \int_{0}^{\infty} \wt{\gamma}(t-\tau) \mu_F(\tau+r) \overline{I}(0,r) \exp\left( - \int_{r}^{\tau+r} \mu_F(s) ds \right) \\ \times \exp\left( - \int_{\tau}^{t} \wt{\gamma}(s-\tau) \overline{\mathfrak{F}}(s) ds \right) dr d\tau  \\
		+ \int_{0}^{\infty} \wt{\gamma}(t+\tau) \overline{R}(0,\tau) \exp\left( - \int_{0}^{t} \wt{\gamma}(s+\tau) \overline{\mathfrak{F}}(s) ds \right) d\tau\,. 
	\end{multline*}
	Note that, concerning the second term, we have to condition upon $\xi$ in order to compute its expression. 
	
	We denote by $\overline{S}(t)$ the proportion of susceptible individuals who have never been infected, $\overline{I}(t)$ the proportion of infected individuals and $\overline{R}(t)$ the proportion of recovered individuals at time $t$.
	
	We obtain the following expression of the limit $(\overline{\mathfrak S},\overline{\mathfrak F})$. 
	
	\begin{prop}We assume that $\|\mu_F\|_\infty<\infty$.
		For the generalized SIRS model described above, the limit $(\overline{\mathfrak S},\overline{\mathfrak F})$ from Theorem~\ref{thm-FLLN} solves the following system of integral equations: 
		\begin{align} \label{eqn-mfkS-SIRS}
			\overline{\mathfrak{S}}(t)&= \overline{S}(0)  \exp \left( - \int_{0}^{t}  \overline{\mathfrak F}(r) dr \right)  \\
			& \quad + \int_{0}^{t}\int_{0}^{\infty}\wt{\gamma}(t-\tau)\frac{f(\tau+r)}{F^c(\tau)}\overline{I}(0,r) \exp\left(-\int_{\tau}^{t}\wt{\gamma}(s-\tau)\overline{\mathfrak{F}}(s)ds\right)dr d\tau  \nonumber \\
			& \quad + \int_{0}^{\infty}\wt{\gamma}(t+\tau)\overline{R}(0,\tau)\exp\left(-\int_{0}^{t}\wt{\gamma}(s+\tau)\overline{\mathfrak{F}}(s)ds\right)d\tau \nonumber
			\\
			& \quad +\int_{0}^{t}\Bigg(\int_{r}^{t} \wt{\gamma}(t-\tau)f(\tau-r)
			\exp\left(-\int_{\tau}^{t}\wt{\gamma}(s-\tau)\overline{\mathfrak{F}}(s)ds\right)d\tau\Bigg)\overline{\mathfrak{S}}(r)\overline{\mathfrak{F}}(r)dr,\non
		\end{align}
		and 
		\begin{align}  \label{eqn-mfkF-SIRS}
			\overline{\mathfrak{F}}(t)&=\int_{t}^{+\infty}\wt{\lambda}(\tau)\overline{I}(0,\tau-t)\frac{F^c(\tau)}{F^c(\tau-t)}d\tau +\int_{0}^{t}\wt{\lambda}(t-\tau)\overline{\mathfrak{S}}(\tau)\overline{\mathfrak{F}}(\tau)F^c(t-\tau)d\tau.
		\end{align}
		In addition, 
		\begin{align*}
			\overline{S}(t) &= \overline{S}(0) \exp\left( - \int_{0}^{t} \overline{\mathfrak{F}}(r) dr \right), \\
			\overline{I}(t) &= \int_{0}^{+\infty}\frac{F^c(t+\tau)}{F^c(\tau)}\overline{I}(0,\tau)d\tau +\int_{0}^{t}F^c(t-s)\overline{\mathfrak{S}}(s)\overline{\mathfrak{F}}(s)ds,\\
			\overline{R}(t) &= \int_{0}^{+\infty} \exp\left( - \int_{0}^{t} \wt{\gamma}(\tau + r) \mathfrak{\overline{F}}(r) dr \right) \overline{R}(0,\tau) d\tau \\
			&\qquad+ \int_{0}^{+\infty} \int_{0}^{t} \exp\left( - \int_{s}^{t} \wt{\gamma}(r-s) \overline{\mathfrak{F}}(r) dr \right) \mu_F(\tau + s) \frac{F^c(\tau+s)}{F^c(\tau)} ds \overline{I}(0,\tau) d\tau \\
			&\qquad+ \int_{0}^{t} \int_{s}^{t} \exp \left( - \int_{u}^{t} \wt{\gamma}(r-u) \overline{\mathfrak{F}}(r) dr \right) \mu_F(u-s) \\
			&\hspace{4cm} \times \exp\left( - \int_{0}^{u-s} \mu_F(r) dr \right) du \overline{\mathfrak{S}}(s) \overline{\mathfrak{F}}(s) ds.
		\end{align*}
	\end{prop} 
	
	\medskip
	\subsection{Connection with the Kermack and McKendrick PDE model} \label{subsec:KMcK}
	
	In \cite{inaba2001kermack}, the model introduced by Kermack and McKendrick in \cite{kermack_contributions_1932,kermack_contributions_1933} was reformulated as follows.
	Let $ \bar{S}(t) $ denote the proportion of susceptible individuals who have never been infected, let $ \bar{I}(t,\tau) $ denote the density of infectious individuals with infection-age $ \tau \geq 0 $ and $ \bar{R}(t,\theta) $ the density of recovered individuals with recovery-age $ \theta \geq 0 $.
	Then, given an initial condition $ (\bar{S}(0), \bar{I}(0,\cdot), \bar{R}(0,\cdot)) $ such that
	\begin{equation*}
		\bar{S}(0) + \int_{0}^{\infty} \bar{I}(0,\tau) d\tau + \int_{0}^{\infty} \bar{R}(0,\theta) d\theta = 1,
	\end{equation*}
	the system evolves according to the following set of partial differential equations:
	%
	%
	\begin{equation} \label{KMcK_model}
		\left\lbrace
		\begin{aligned}
			&\frac{d \overline{S}}{d t}(t) = -\overline{S}(t) \int_{0}^{+\infty} \wt{\lambda}(\tau) \overline{I}(t,\tau) d\tau \\
			&\frac{\partial \overline{I}}{\partial t}(t,\tau) + \frac{\partial \overline{I}}{\partial \tau}(t,\tau) = -\mu_F(\tau) \overline{I}(t,\tau) \\
			&\frac{\partial \overline{R}}{\partial t}(t,\tau) + \frac{\partial \overline{R}}{\partial \tau}(t,\tau) = -\overline{R}(t,\tau) \wt{\gamma}(\tau) \int_{0}^{+\infty} \wt{\lambda}(r) \overline{I}(t,r) dr \\
			&\overline{I}(t,0) = \left(\overline{S}(t) + \int_{0}^{+\infty} \wt{\gamma}(\theta) \overline{R}(t,\theta) d\theta \right) \int_{0}^{+\infty} \wt{\lambda} (\tau)\overline{I}(t,\tau)d\tau
			\\
			&\overline{R}(t,0) = \int_{0}^{+\infty} \mu_F(\tau) \overline{I}(t,\tau) d\tau.
		\end{aligned}
		\right.
	\end{equation}
	
	We will use the following notion of a solution of the system of ODE-PDEs \eqref{KMcK_model}. For us, a solution is an element 
	$(\bar{S}(t), \bar{I}(t,\cdot), \bar{R}(t,\cdot), t \geq 0)\in\mathcal{C}\left(\R_+,\R_+\times L^1(\R_+)\times L^1(\R_+)\right)$ satisfying \eqref{KMcK_model}, 
	where the derivatives are taken in the distributional sense, and \eqref{KMcK_model} is understood as follows. If $(\bar{S}(t), \bar{I}(t,\cdot), \bar{R}(t,\cdot)$ is a solution, then the continuous function $\bar{S}(t)$ is the unique solution of a linear equation which is given explicitly in terms of $\bar{I}$ and $\wt\lambda$, and the equation is satisfied pointwise. It follows from the second equation that $s\mapsto\bar{I}(t+s,s)$ also solves a linear  equation,  whose explicit solution is at least continuous, hence we can define $\bar{I}(t,0)=\lim_{s\to0}\bar{I}(t+s,s)$, and it is this quantity which is assumed to satisfy the boundary condition stated on the fourth line of  \eqref{KMcK_model}. A similar argument can be applied to $\bar{R}$. 
	Note that it is essential that $\wt{\lambda}$ and $\wt{\gamma}$ be bounded, as well as $\mu_F$. This last point will be an assumption in our statements below.
	The classical reference for systems of PDEs of the type of \eqref{KMcK_model} is the book of Webb \cite{webb1985theory}, which establishes existence and uniqueness, but with a notion of solution which is a bit tricky, and we prefer to avoid presenting it. 
	We will show below how to associate to any solution of \eqref{KMcK_model} the solution to the system of integral equations \eqref{eqn-mfkS-SIRS}-\eqref{eqn-mfkF-SIRS},
	and vice versa how to construct a solution to \eqref{KMcK_model} from the solution of \eqref{eqn-mfkS-SIRS}-\eqref{eqn-mfkF-SIRS}. 
	This establishes the connection between our results, in the particular case of the model described in the present section, and those of Kermack and McKendrick.
	We shall then deduce existence and uniqueness of a solution of \eqref{KMcK_model} as a corollary.
	
	\begin{prop}\label{Th-kmk}
		We assume that $\|\mu_F\|_\infty<\infty$. If  $ (\bar{S}(t), \bar{I}(t,\cdot), \bar{R}(t,\cdot), t \geq 0) $ is a solution of the system \eqref{KMcK_model}, and $ (\overline{\mathfrak{S}}(\cdot), \overline{\mathfrak{F}}(\cdot)) $ is given by
		\begin{equation} \label{eqn-mfk-SIRS-two}
			\left\{
			\begin{aligned}
				\overline{\mathfrak{S}}(t)&=\overline{S}(t)+\int_{0}^{+\infty}\wt{\gamma}(\tau)\overline{R}(t,\tau)d\tau, \\
				\overline{\mathfrak{F}}(t)&=\int_{0}^{+\infty}\wt{\lambda}(\tau)\overline{I}(t,\tau)d\tau\,,
			\end{aligned}\right.
		\end{equation}
		then $ (\overline{\mathfrak{S}}(\cdot), \overline{\mathfrak{F}}(\cdot)) $ solves \eqref{eqn-mfkS-SIRS}-\eqref{eqn-mfkF-SIRS}.
		Conversely, given  $ (\overline{\mathfrak{S}}(\cdot), \overline{\mathfrak{F}}(\cdot)) $ the unique solution of \eqref{eqn-mfkS-SIRS}-\eqref{eqn-mfkF-SIRS}, the following is a solution to \eqref{KMcK_model}.
		\begin{equation} \label{PDE_from_Volterra}
			\left\lbrace
			\begin{aligned}
				\overline{S}(t) &=\overline{S}(0)\exp\left(-\int_{0}^{t}\overline{\mathfrak{F}}(s)ds\right),\\
				\overline{I}(t,\tau)&=\overline{I}(0,\tau-t)\frac{F^c(\tau)}{F^c(\tau-t)}\indic{\tau>t} + \overline{\mathfrak F}(t-\tau)\overline{\mathfrak S}(t-\tau)F^c(\tau)\indic{t\geq\tau},\\
				\overline{R}(t,\tau)&=\overline{R}(t-\tau,0)\exp\left(-\int_{0}^{\tau}\wt{\gamma}(s)\overline{\mathfrak{F}}(t+s-\tau)ds\right)\indic{\tau< t} \\
				&\hspace{1cm} +\overline{R}(0,\tau-t)\exp\left(-\int_{\tau-t}^{\tau}\wt{\gamma}(s)\overline{\mathfrak{F}}(t-\tau+s)ds\right)\indic{\tau\geq t},\\
				\overline{R}(t,0)&=\int_{0}^{+\infty}\overline{I}(0,\tau)\frac{f(t+\tau)}{F^c(\tau)}d\tau +\int_{0}^{t}f(t-\tau)\overline{\mathfrak{S}}(\tau)\overline{\mathfrak{F}}(\tau)d\tau,
			\end{aligned}
			\right.
		\end{equation}
		with $F$ given by \eqref{N2} and $f$ is the density of $F$. 
	\end{prop} 
	
	\begin{proof}
		Assuming that $ (\bar{S}(t), \bar{I}(t,\cdot), \bar{R}(t,\cdot), t \geq 0) $ is a solution of the system \eqref{KMcK_model},  we now show that $(\overline{\mathfrak{S}}(\cdot), \overline{\mathfrak{F}}(\cdot))$ given by \eqref{eqn-mfk-SIRS-two}
		satisfies \eqref{eqn-mfkS-SIRS} and \eqref{eqn-mfkF-SIRS}.
		Integrating $\overline{I}$ and $\overline{R}$ along the characteristics, which are the parallels of the first diagonal of 
		$\R_+^2$, we obtain:
		\begin{equation}\label{KMcK_model-chara}
			\left\lbrace
			\begin{aligned}
				&\overline{S}(t)=\overline{S}(0)\exp\left(-\int_{0}^{t}\int_{0}^{+\infty}\wt{\lambda}(\tau)\overline{I}(s,\tau)d\tau ds\right),\\\\
				&\overline{I}(t,\tau)=\overline{I}(0,\tau-t)\frac{F^c(\tau)}{F^c(\tau-t)}\indic{\tau>t} + \overline{I}(t-\tau,0)F^c(\tau)\indic{t\geq\tau},\\\\
				&\overline{R}(t,\tau)=\overline{R}(t-\tau,0)\exp\left(-\int_{0}^{\tau}\wt{\gamma}(s)\int_{0}^{+\infty}\wt{\lambda}(r)\overline{I}(t+s-\tau,r)drds\right)\indic{\tau< t}   \\
				&\hspace{2cm} +\overline{R}(0,\tau-t)\exp\left(-\int_{\tau-t}^{\tau}\wt{\gamma}(s)\int_{0}^{+\infty}\wt{\lambda}(r)\overline{I}(t-\tau+s,r)drds\right)\indic{\tau\geq t},\\\\
				&\overline{I}(t,0) = \left(\overline{S}(t) + \int_{0}^{+\infty} \wt{\gamma}(\theta) \overline{R}(t,\theta) d\theta \right) \int_{0}^{+\infty} \wt{\lambda} (\tau)\overline{I}(t,\tau)d\tau,
				\\\\
				&\overline{R}(t,0)=\int_{0}^{+\infty}\overline{I}(0,\tau)\frac{f(t+\tau)}{F^c(\tau)}d\tau +\int_{0}^{t}f(t-\tau)\overline{I}(\tau,0)d\tau.
			\end{aligned}
			\right.
		\end{equation}
		By \eqref{eqn-mfk-SIRS-two}, this can be written
		\begin{equation*}
			\left\lbrace
			\begin{aligned}
				&\overline{S}(t)=\overline{S}(0)\exp\left(-\int_{0}^{t} \overline{\mathfrak{F}}(s) ds\right),\\\\
				&\overline{I}(t,\tau)=\overline{I}(0,\tau-t)\frac{F^c(\tau)}{F^c(\tau-t)}\indic{\tau>t} + \overline{I}(t-\tau,0)F^c(\tau)\indic{t\geq\tau},\\\\
				&\overline{R}(t,\tau)=\overline{R}(t-\tau,0)\exp\left(-\int_{0}^{\tau}\wt{\gamma}(s)\overline{\mathfrak{F}}(t+s-\tau)ds\right)\indic{\tau< t}   \\
				&\hspace{2cm} +\overline{R}(0,\tau-t)\exp\left(-\int_{\tau-t}^{\tau}\wt{\gamma}(s)\overline{\mathfrak{F}}(t-\tau+s)ds\right)\indic{\tau\geq t},\\\\
				&\overline{I}(t,0) = \overline{\mathfrak{S}}(t) \overline{\mathfrak{F}}(t)
				\\\\
				&\overline{R}(t,0)=\int_{0}^{+\infty}\overline{I}(0,\tau)\frac{f(t+\tau)}{F^c(\tau)}d\tau +\int_{0}^{t}f(t-\tau)\overline{I}(\tau,0)d\tau,
			\end{aligned}
			\right.
		\end{equation*}
		Substituting those expressions for $ \overline{S}(t) $, $ \overline{R}(t,\tau) $ and $ \overline{I}(t,\tau) $ in  \eqref{eqn-mfk-SIRS-two} using the above expressions yields \eqref{eqn-mfkS-SIRS} and \eqref{eqn-mfkF-SIRS}.	
		
		Conversely let $ (\overline{\mathfrak{S}}(\cdot), \overline{\mathfrak{F}}(\cdot)) $ be a solution of the set of equations \eqref{eqn-mfkS-SIRS}-\eqref{eqn-mfkF-SIRS}. We now show that $ (\bar{S}(t), \bar{I}(t,\cdot), \bar{R}(t,\cdot), t \geq 0) $ given by \eqref{PDE_from_Volterra} solves the system \eqref{KMcK_model}. It suffices in fact to show that $ (\bar{S}(t), \bar{I}(t,\cdot), \bar{R}(t,\cdot), t \geq 0)$ satisfies \eqref{KMcK_model-chara} (see Theorem~2.2 in \cite{webb1985theory}).
		Comparing \eqref{PDE_from_Volterra} and \eqref{KMcK_model-chara}, we see that it is enough to show that, if $ \overline{I}(t,\tau) $, $ \overline{S}(t) $ and $ \overline{R}(t,\tau) $ are given by \eqref{PDE_from_Volterra}, then
		\begin{equation} \label{recover_F}
			\int_{0}^{+\infty} \wt{\lambda}(\tau) \overline{I}(t,\tau) d\tau = \overline{\mathfrak{F}}(t),
		\end{equation}
		and
		\begin{equation} \label{recover_S}
			\overline{S}(t) + \int_{0}^{+\infty} \wt{\gamma}(\tau) \overline{R}(t,\tau) d \tau = \overline{\mathfrak{S}}(t).
		\end{equation}
		Indeed, replacing $ \overline{I}(t,\tau) $, $ \overline{S}(t) $ and $ \overline{R}(t,\tau) $ by their expressions in \eqref{PDE_from_Volterra}, we obtain
		\begin{multline*}
			\int_{0}^{+\infty} \wt{\lambda}(\tau) \overline{I}(t,\tau) d\tau = \int_{0}^{+\infty} \wt{\lambda}(\tau+t) \overline{I}(0,\tau) \frac{F^c(\tau+t)}{F^c(\tau)} d\tau \\ + \int_{0}^{t} \wt{\lambda}(\tau) F^c(\tau) \overline{\mathfrak{F}}(t-\tau) \overline{\mathfrak{S}}(t-\tau) d\tau,
		\end{multline*}
		and
		\begin{multline*}
			\overline{S}(t) + \int_{0}^{+\infty} \wt{\gamma}(\tau) \overline{R}(t,\tau) d\tau = \overline{S}(0) \exp\left( -\int_{0}^{+\infty} \overline{\mathfrak{F}}(s) ds \right) \\+ \int_{0}^{+\infty} \wt{\gamma}(\tau+t) \overline{R}(0,\tau) \exp\left( -\int_{\tau}^{\tau+t} \wt{\gamma}(s) \overline{\mathfrak{F}}(s-\tau) ds \right) d\tau \\+ \int_{0}^{t} \wt{\gamma}(\tau)  \exp\left( -\int_{0}^{\tau} \wt{\gamma}(s) \overline{\mathfrak{F}}(t+s-\tau) ds \right) \overline{R}(t-\tau,0) d\tau.
		\end{multline*}
		Replacing $ \overline{R}(t-\tau,0) $ by its expression in \eqref{PDE_from_Volterra} and using the fact that $ (\overline{\mathfrak{S}}, \overline{\mathfrak{F}}) $ solves \eqref{eqn-mfkS-SIRS}-\eqref{eqn-mfkF-SIRS}, we obtain \eqref{recover_F} and \eqref{recover_S}.
		
		This completes the proof of the equivalence. 
	\end{proof}
	
	\begin{coro}
		We assume that $\|\mu_F\|_\infty<\infty$.  Then the system \eqref{KMcK_model} has a unique solution.
	\end{coro}
	\begin{proof}
		From the solution of  \eqref{eqn-mfkS-SIRS}-\eqref{eqn-mfkF-SIRS}, we define a solution of \eqref{KMcK_model} with the help of \eqref{PDE_from_Volterra}, which implies the existence of a solution to \eqref{KMcK_model}. 
		Now let $(\bar{S}^1,\bar{I}^1, \bar{R}^1)$ and 
		$(\bar{S}^2,\bar{I}^2, \bar{R}^2)$ be two solutions of \eqref{KMcK_model} with the same initial condition at time $t=0$. Through \eqref{eqn-mfk-SIRS-two}, they define the same solution of  \eqref{eqn-mfkS-SIRS}-\eqref{eqn-mfkF-SIRS}. Hence for all $t\ge0$,
		\[ \int_0^\infty \wt{\lambda}(\tau)\bar{I}^1(t,\tau)d\tau=\int_0^\infty \wt{\lambda}(\tau)\bar{I}^2(t,\tau)d\tau.\]
		Consequently from the first line of \eqref{PDE_from_Volterra}, $\bar{S}^1(t)=\bar{S}^2(t)$, for all 
		$t\ge0$. This, combined with the first line of \eqref{eqn-mfk-SIRS-two},  implies that for all $t\ge0$,
		\[\int_0^\infty\wt{\gamma}(\tau)\bar{R}^1(t,\tau)d\tau=\int_0^\infty\wt{\gamma}(\tau)\bar{R}^2(t,\tau)d\tau\,.\]
		The above three facts combined with the fourth line of \eqref{PDE_from_Volterra} imply that $\bar{I}^1(t,0)=\bar{I}^2(t,0)$
		for all $t\ge0$, and hence also $\bar{R}^1(t,0)=\bar{R}^2(t,0)$ for all $t\ge0$. The second and third lines of \eqref{PDE_from_Volterra}
		allow us to conclude that $\bar{I}^1(t,\tau)=\bar{I}^2(t,\tau)$ and $\bar{R}^1(t,\tau)=\bar{R}^2(t,\tau)$ for all $t,\tau\ge0$.
		Hence uniqueness.
	\end{proof}

	\bigskip 
	
	\section{Proofs for the FLLN} \label{sec-proofs-FLLN}
	
	In this section, we prove the FLLN. We start with the proof of Lemma~\ref{rem-1}.
	
	\begin{proof}[Proof of Lemma~\ref{rem-1}]
		We can note that, if we multiply the equation for $x(t)$ by $y(t)$, we obtain
		\begin{multline*}
			y(t) x(t) = \E \left[ y(t) \gamma_0(t) \exp \left( -\int_0^t \gamma_0(r) y(r) dr \right) \right] \\+ \int_0^t \E \left[ y(t) \gamma(t-s) \exp \left( -\int_s^t \gamma(r-s) y(r) dr \right) \right] x(s) y(s) ds.
		\end{multline*}
		As a result,
		\begin{multline*}
			\frac{d}{dt} \left( \E \left[ \exp \left( - \int_{0}^{t} \gamma_0(r) y(r) dr \right)\right] \right.\\ \left. + \int_{0}^{t} \E \left[ \exp \left( - \int_{s}^{t} \gamma(r-s) y(r) dr \right) \right] x(s) y(s) ds \right) = 0.
		\end{multline*}
		Note that this function may not be differentiable everywhere since there may be an at most countable set of points where its right and left derivatives do not coincide. Nevertheless this function is absolutely continuous and since its derivative is equal to zero almost everywhere, it is constant.   
		
		Hence, integrating between 0 and t, we obtain the result.
	\end{proof}
	We next prove Theorem~\ref{ExistUniq}.
	\begin{proof}[Proof of Theorem~\ref{ExistUniq}] 
		Let $(x,y)\in D^2_+$ be a solution to the set of equation \eqref{LLN_xy}. 
		From Assumption~\ref{AS-lambda}, $\gamma_0(t)\leq1$ and $\gamma(t-s)\leq1$. Hence the combination of \eqref{LLN_xy} and \eqref{eqG1} implies that $x(t)\leq1.$ Moreover, from Assumption \ref{AS-lambda} if $\lambda(t)>0 \,(\,\lambda_{0}(t)>0)$ then $\gamma(s)=0\, (\,\gamma_0(s)=0)$ for $0\leq s\leq t$, we have
		\begin{align*}
			y(t)&=\overline{I}(0)\overline{\lambda}_0(t)+\int_{0}^{t}\overline{\lambda}(t-s)x(s)y(s)ds\\
			&=\E\left[\lambda_{0}(t)\exp\left(-\int_0^t \gamma_0(r) y(r) dr \right)\right]\\
			&\hspace{2cm}+\int_{0}^{t}\E\left[\lambda(t-s)\exp\left(-\int_s^t \gamma(r-s) y(r) dr \right)\right]x(s)y(s)ds\\
			&\leq\lambda^\ast,
	\end{align*}
	where the last inequality follows from \eqref{eqG1} and the fact that $\lambda_0(t), \lambda(t-s)\le\lambda^\ast$.
	Consequently, if $(x,y)$ solves \eqref{LLN_xy}, then for $t\geq0,\,x(t)\leq1$ and $y(t)\le \lambda^\ast$.
	
	Suppose now that we have two solutions $(x^1,y^1)$ and $(x^2,y^2)$ of equations \eqref{LLN_xy}. Then we have 
	\begin{align*}
		x^1(t)-x^2(t)&=\E\left[ \gamma_0(t)\left( \exp\left(-\int_0^t \gamma_0(r) y^1(r) dr \right)-\exp\left(-\int_0^t \gamma_0(r) y^2(r) dr \right)\right)\right]\\
		&\qquad+\int_0^t \E\Bigg[ \gamma(t-s)\Bigg( \exp\left(- \int_s^t \gamma(r-s) y^1(r) dr \right)x^1(s) y^1(s)\\
		&\qquad \qquad \qquad -\exp\left(- \int_s^t \gamma(r-s) y^2(r) dr \right)x^2(s) y^2(s)\Bigg)\Bigg]ds\,,\\
		y^1(t)-y^2(t)&=\int_0^t\overline{\lambda}(t-s)\left[x^1(s) y^1(s)-x^2(s) y^2(s)\right]ds\,.
	\end{align*}
	We first deduce from the second relation, taking into account that for $i=1,2$, $x^i(s)\le 1$ and 
	$y^i(s)\le\lambda^\ast$, that for any $t>0$, 
	\begin{align}\label{estim1}
		\left| y^1(t)-y^2(t)\right|&\le\lambda^\ast \int_0^t\left[\left|x^1(s)-x^2(s)\right|\left|y^1(s)\right|+
		\left| y^1(s)-y^2(s) \right|\left|x^2(s)\right|\right]ds\nonumber
		\\&\leq \lambda^\ast\max(\lambda^\ast,1)\int_0^t\left[\left|x^1(s)-x^2(s)\right|+\left| y^1(s)-y^2(s) \right|\right]ds. 
	\end{align}
	We now exploit the first relation. Since $\left|\exp(-a)-\exp(-b)\right|\leq |a-b|,\,\forall a,b\in\mathbb{R}_+,\, \gamma\leq1$ and $x^1y^1\leq\lambda_*$, we have for $0\leq t\leq T$,
	\begin{align*}
		\left|x^1(t)-x^2(t)\right|
		&\leq\int_0^t\left|y^1(s)-y^2(s)\right|ds\\
		&\qquad+\int_0^t \E\Bigg[\Bigg| \exp\left(- \int_s^t \gamma(r-s) y^1(r) dr \right)x^1(s) y^1(s)\\
		&\qquad\qquad \qquad -\exp\left(- \int_s^t \gamma(r-s) y^2(r) dr \right)x^2(s) y^2(s)\Bigg|\Bigg]ds
		\\&\leq\int_0^t\left|y^1(s)-y^2(s)\right|ds\\
		&\qquad
		+\int_0^t\E\Bigg[\bigg|\exp\left(- \int_s^t \gamma(r-s) y^1(r) dr \right) \\
		&\qquad \qquad \qquad \qquad \qquad  -\exp\left(- \int_s^t \gamma(r-s) y^2(r) dr \right)\bigg|\Bigg]x^1(s) y^1(s)ds\\
		&\qquad+\int_{0}^{t}\E\left[\exp\left(- \int_s^t \gamma(r-s) y^2(r) dr \right)\right]\left|x^1(s)y^1(s)-x^2(s)y^2(s)\right|ds\\
		&\leq\int_0^t\left|y^1(s)-y^2(s)\right|ds
		+T\lambda_*\int_0^t\left|y^1(s)-y^2(s)\right|ds\\
		&\qquad+\int_{0}^{t}\left|x^1(s)y^1(s)-x^2(s)y^2(s)\right|ds.
	\end{align*}
	Hence, again, given $T>0$, there exists a constant $C$ such that for any $0\le t\le T$,
	\begin{align}\label{estim2}
		\left|x^1(t)-x^2(t)\right|&\le 
		C\int_0^t\left[\left|x^1(s)-x^2(s)\right|
		+\left| y^1(s)-y^2(s) \right|\right]ds\,.
	\end{align}
	Uniqueness follows from \eqref{estim1}, \eqref{estim2} and the Gronwall Lemma.
	
	Now local existence follows by an approximation procedure, which exploits the estimates \eqref{estim1} and \eqref{estim2}. Indeed, Picard iteration procedure works for an existence of integral equation as for  the more technical ODEs. For detail, see Brunner \cite{brunner_volterra_2017}. Global existence then follows 
	from the above a priori estimates, which forbids explosion. 
	
	Note that, it easy to check that the map 
	\begin{multline*}
		t\mapsto\left(\int_{0}^{t}\E\left[\gamma(t-s)\exp\left(-\int_s^t \gamma(r-s) \overline{\mathfrak{F}}(r) dr \right)\right]\overline{\mathfrak{S}}(s)\overline{\mathfrak{F}}(s)ds, \right. \\ \left. \int_{0}^{t}\overline{\lambda}(t-s)\overline{\mathfrak{S}}(s)\overline{\mathfrak{F}}(s)ds\right),
	\end{multline*}
	is continuous.
	Therefore, to conclude the continuity of the map $t\mapsto\left(\overline{\mathfrak{S}}(t),\overline{\mathfrak{F}}(t)\right)$ solution of equations \eqref{LLN_xy}, it remains to show that the map
	\begin{equation}t\mapsto \left(\E\left[\gamma_0(t)\exp\left(-\int_0^t \gamma_0(r) \overline{\mathfrak{F}}(r) dr \right)\right],\overline{\lambda}_0(t)\right),\label{eq-init-c}\end{equation}
	is continuous.
	Indeed, as $\gamma_0$ has bounded variation, there exist two non-decreasing functions $\gamma_0^1$ and $\gamma_0^2$ such that $\gamma_0=\gamma_0^1-\gamma_0^2.$ Thus for $t\geq t_0,$ as $\gamma_0\leq1,\,\overline{\mathfrak{F}}\leq\lambda_*$,
	\begin{multline*}
		\left|\E\left[\gamma_0(t)\exp\left(-\int_{0}^{t}\gamma_{0}(r)\overline{\mathfrak{F}}(r)dr\right)\right]-\E\left[\gamma_0(t_0)\exp\left(-\int_{0}^{t_0}\gamma_{0}(r)\overline{\mathfrak{F}}(r)dr\right)\right]\right|\\
		\leq\E\left[\left|\gamma_0^1(t)-\gamma_0^1(t_0)\right|\right]+2\E\left[\int_{t_0}^{t}\gamma_0(r)\overline{\mathfrak{F}}(r)dr\right]+\E\left[\left|\gamma_0^2(t)-\gamma_0^2(t_0)\right|\right]\\
		=\left|\E\left[\gamma_0^1(t)\right]-\E\left[\gamma_0^1(t_0)\right]\right|+2\lambda_*|t-t_0|+\left|\E\left[\gamma_0^2(t)\right]-\E\left[\gamma_0^2(t_0)\right]\right|\\
		\leq2\left|\E\left[\gamma_0(t)\right]-\E\left[\gamma_0(t_0)\right]\right|+2\lambda_*|t-t_0|,
	\end{multline*}
	where the third line follows from the fact that $\gamma_0^1$ and $\gamma_0^2$ are non-decreasing.
	Hence as the map $t\mapsto \left(\E\left[\gamma_0(t)\right],\overline{\lambda}_0(t)\right)$ is continuous, it follows that \eqref{eq-init-c} is continuous. Therefore, the map $t\mapsto\left(\overline{\mathfrak{S}}(t),\overline{\mathfrak{F}}(t)\right)$ is continuous. Theorem \ref{ExistUniq} is established.
\end{proof} 



\medskip

We next prove Theorem~\ref{thm-FLLN}. We first construct a system of stochastic equations driven by Poisson random measures (PRMs), and then use an approach of the type of propagation of chaos as in \cite{sznitman1991topics}. 

Take $m\in D_+$, let $(\lambda_0,\gamma_0)$ be a random variable taking values in $D_+^2$ and $(\lambda_i,\gamma_i)_{i\ge 1}$ be an independent collection of i.i.d. random variables taking values in $D_+^2$.
Also let $Q$ be a standard Poisson random measure on $\R^2_+$, independent of the previous random variables. 
We define for $t\geq0$, 
\begin{equation}
	\left\{
	\begin{aligned}
		A^{(m)}(t)&:=\int_{0}^{t}\int_{0}^{+\infty}\indic{u\leq\Upsilon^{(m)}(r^-)}Q(dr,du)\\
		\Upsilon^{(m)}(t)&:=\gamma_{A^{(m)}(t)}(\varsigma^{(m)}(t))m(t)\label{suite_A}
	\end{aligned}\right.
\end{equation}
where $\varsigma^{(m)}$ is defined in the same manner as $\varsigma^N_1$ with $ A^{(m)} $ instead of $ A^N_1$.

Let 
$$\overline{\mathfrak{F}}^{(m)}(t)=\mathbb{E}\left[\lambda_{A^{(m)}(t)}(\varsigma^{(m)}(t))\right], \quad \text{and} \quad\overline{\mathfrak{S}}^{(m)}(t)=\mathbb{E}\left[\gamma_{A^{(m)}(t)}(\varsigma^{(m)}(t))\right].$$

\begin{lemma}\label{exist_A}
	There exists a unique function $m^*\in D_+$ such that
	\[\overline{\mathfrak{F}}^{(m^*)}=m^*.\] 
	Moreover, $(\overline{\mathfrak S}^{(m^\ast)},\overline{\mathfrak{F}}^{(m^*)})$ solves the set of equations \eqref{LLN_xy}.
\end{lemma}

\begin{proof}	
	Let us denote the jump times of the process $A^{(m)}$ by $(\tau_{i}^{(m)})_{i\geq0}$ with $\tau^{(m)}_0=0$. 	From Assumption~\ref{AS-lambda}, for $t\geq\tau_i^{(m)},$ if $\lambda_i(t-\tau_i^{(m)})\neq0,$ we have $$\int_{\tau_i^{(m)}}^{t}\gamma_i(r-\tau_i^{(m)})\overline{\mathfrak F}^{(m)}(r)dr=0,$$
	hence $A^{(m)}(t)=A^{(m)}(\tau_{i}^{(m)})=i$.
	By contradiction this also implies that $ \lambda_j(t-\tau_j^{(m)}) = 0 $ for all $ 0 \leq j \leq i-1 $.
	As a result,
	\begin{equation*}
		\lambda_{A^{(m)}(t)}(\varsigma^{(m)}(t)) = \lambda_{0}(t)+\sum_{i=1}^{A^{(m)}(t)}\lambda_{i}(t-\tau_{i}^{(m)}),
	\end{equation*}
	for all $ t \geq 0 $, almost surely.
	We thus obtain that
	\begin{align*}
		\overline{\mathfrak{F}}^{(m)}(t)&=\mathbb{E}\left[\lambda_{A^{(m)}(t)}(\varsigma^{(m)}(t))\right]\\
		&=\E\left[\lambda_{0}(t)+\sum_{i=1}^{A^{(m)}(t)}\lambda_{i}(t-\tau_{i}^{(m)})\right]
	\end{align*}
	Also recall that, by the definition of $ \eta_0 $, $ \lambda_0(t) = \lambda_0(t) \indic{\eta_0 > 0} $, and that $\overline{\lambda}_0(t) = \E\big[\lambda_{0}(t)\big|\eta_0>0\big]$ and $\overline{I}(0)=\mathbb{P}(\eta_0>0)$.
	Therefore,
	\begin{align*}
		\overline{\mathfrak{F}}^{(m)}(t)&=\E\left[\lambda_0(t)\left|\eta_0>0\right.\right]\mathbb{P}\left(\eta_0>0\right)+\sum_{i\geq1}\E\left[\lambda_{i}(t-\tau_{i}^{(m)})\indic{\tau_{i}^{(m)}\leq t}\right]\\
		&=\overline{I}(0)\overline{\lambda}_0(t)+\sum_{i\geq1}\E\left[\lambda_{i}(t-\tau_{i}^{(m)})\indic{\tau_{i}^{(m)}\leq t}\right],
	\end{align*}
	and as $\lambda_{i}$ and $\tau_{i}^{(m)}$ are independent, we further obtain that 
	\begin{align*}
		\overline{\mathfrak{F}}^{(m)}(t)&=\overline{I}(0)\overline{\lambda}_0(t)+\sum_{i\geq1}\E\Big[\overline{\lambda}(t-\tau_{i}^{(m)})\indic{\tau_{i}^{(m)}\leq t}\Big]\nonumber\\
		&=\overline{I}(0)\overline{\lambda}_0(t)+\E\left[\int_{0}^t\overline{\lambda}(t-s)A^{(m)}(ds)\right].
	\end{align*}
	Note that the process $t \mapsto A^{(m)}(t)-\int_{0}^{t}\Upsilon^{(m)}(s)ds$ is a martingale. 
	The integral of a bounded deterministic  function over a compact interval with respect to such a martingale has zero expectation, hence
	\begin{equation} \label{zero_expectation}
		\E\left[\int_{0}^{t}\overline{\lambda}(t-s)\left(A^{(m)}(ds)-\Upsilon^{(m)}(s)ds\right)\right]=0.
	\end{equation}
	As a result,	
	\begin{equation}\label{eqn0}
		\overline{\mathfrak{F}}^{(m)}(t)=\overline{I}(0)\overline{\lambda}_0(t)+\int_{0}^t\overline{\lambda}(t-s)m(s)\overline{\mathfrak{S}}^{(m)}(s)ds.
	\end{equation}
	In addition, 
	\begin{equation*}
		\gamma_{A^{(m)}(t)}(\varsigma^{(m)}(t))=\gamma_{0}(t)\indic{A^{(m)}(t)=0}+\sum_{i=1}^{+\infty}\gamma_{i}(t-\tau_{i}^{(m)})\indic{\tau_{i}^{(m)}\leq t}\indic{A^{(m)}(t)=i}\,\, .
	\end{equation*}
	
	Let 
	\[\mathcal{F}_t=\sigma\left(\left\{(\lambda_{i})_{0\leq i\leq A^{(m)}(t)},(\gamma_{i})_{0\leq i\leq A^{(m)}(t)},A^{(m)}(t'),0\leq t'\leq t\right\}\right).\] 
	Since $Q_{\big|_{]\tau^{(m)}_i,t]}}$ is independent of $\mathcal{F}_{\tau_{i}^{(m)}}$ (setting $ ]\tau^{(m)}_i, t] = \emptyset $ when $ \tau^{(m)}_i \geq t $),  we have 
	\begin{align*}
		&\mathbb{P}\left(A^{(m)}(t)=i\Big|\mathcal{F}_{\tau_{i}^{(m)}}\right)\mathds{1}_{\tau^{(m)}_i\leq t} \\
		&=\mathbb{P}\left(Q\left(\{(s,u)\in\R^2_+,\,\tau_{i}^{(m)}< s\le t,\,\gamma_{i}((s-\tau_{i}^{(m)})^-)m(s^-)\geq u\}\right)=0\Big|\mathcal{F}_{\tau_{i}^{(m)}}\right)\mathds{1}_{\tau^{(m)}_i\leq t}\\
		&=\exp\left(-\int_{\tau_{i}^{(m)}}^{t}\gamma_{i}(r-\tau_{i}^{(m)})m(r)dr\right)\mathds{1}_{\tau^{(m)}_i\leq t}.
	\end{align*}
	Thus, since $\gamma_0$ and $Q$ are also independent,  we obtain 
	\begin{align*}
		\overline{\mathfrak{S}}^{(m)}(t)&=\E\left[\gamma_{0}(t)\indic{A^{(m)}(t)=0}\right]+\sum_{i\geq1}\E\left[\gamma_{i}(t-\tau_{i}^{(m)})\indic{\tau_{i}^{(m)}\leq t}\mathbb{P}\left(A^{(m)}(t)=i\Big|\mathcal{F}_{\tau_{i}^{(m)}}\right)\right]\\
		&=\E\left[\gamma_{0}(t)\exp\left(-\int_{0}^{t}\gamma_{0}(r)m(r)dr\right)\right] \\
		& \quad +\sum_{i\geq1}\E\left[\gamma_{i}(t-\tau_{i}^{(m)})\indic{\tau_{i}^{(m)}\leq t}\exp\left(-\int_{\tau_{i}^{(m)}}^{t}\gamma_{i}(r-\tau_{i}^{(m)})m(r)dr\right)\right]\,. 
	\end{align*}
	Moreover, since $\gamma_i$ and  $\tau_{i}^{(m)}$ are independent, recalling that the law of $\gamma$ is denoted by $\mu$,  we further obtain 
	\begin{multline*}
		\overline{\mathfrak{S}}^{(m)}(t)=\E\left[\gamma_{0}(t)\exp\left(-\int_{0}^{t}\gamma_{0}(r)m(r)dr\right)\right] \\ +\sum_{i\geq1}\E\left[\int_{D}\gamma(t-\tau_{i}^{(m)})\exp\left(-\int_{\tau_{i}^{(m)}}^{t}\gamma(r-\tau_{i}^{(m)})m(r)dr\right)\mu(d\gamma)\indic{\tau_{i}^{(m)}\leq t}\right]\,. 
	\end{multline*}
	
	In addition, as $(\tau_i^{(m)})_i$ are the jump time of $A^{(m)}$, by Fubini's theorem, we have 
	\begin{align}\label{eqn2}
		\overline{\mathfrak{S}}^{(m)}(t)&=\E\left[\gamma_{0}(t)\exp\left(-\int_{0}^{t}\gamma_{0}(r)m(r)dr\right)\right]\nonumber\\&\hspace*{1cm}+\int_{D}\E\left[\sum_{i=1}^{A^{(m)}(t)}\gamma(t-\tau_{i}^{(m)})\exp\left(-\int_{\tau_{i}^{(m)}}^{t}\gamma(r-\tau_{i}^{(m)})m(r)dr\right)\right]\mu(d\gamma)\nonumber\\
		&=\E\left[\gamma_{0}(t)\exp\left(-\int_{0}^{t}\gamma_{0}(r)m(r)dr\right)\right]\nonumber\\&\hspace*{1cm}+\int_{D}\E\left[\int_{0}^t\gamma(t-s)\exp\left(-\int_{s}^{t}\gamma(r-s)m(r)dr\right)A^{(m)}(ds)\right]\mu(d\gamma).
	\end{align}
	For any given $\gamma\in D$, we define 
	\[g_\gamma(s,t):=\gamma(t-s)\exp\left(-\int_{s}^{t}\gamma(r-s)m(r)dr\right).\]
	Since \[A^{(m)}(t)=\int_{0}^{t}\int_{0}^{+\infty}\indic{u\le \Upsilon^{(m)}(s^-)}Q(ds,du),\]
	by the same argument as the one used in \eqref{zero_expectation} we obtain 
	\begin{align}\label{eqn3}
		\E\left[\int_{0}^{t}g_\gamma(s,t)A^{(m)}(ds)\right]
		&=\E\left[\int_{0}^{t}g_\gamma(s,t)\Upsilon^{(m)}(s)ds\right]\nonumber\\
		&=\int_{0}^{t}g_\gamma(s,t)\E\left[\Upsilon^{(m)}(s)\right]ds\nonumber\\
		&=\int_{0}^{t}g_\gamma(s,t)\overline{\mathfrak{S}}^{(m)}(s)m(s)ds.
	\end{align} 
	Then, from \eqref{eqn2} and \eqref{eqn3},  we deduce that
	\begin{align}\label{eqn4}
		\overline{\mathfrak{S}}^{(m)}(t)
		&=\E\left[\gamma_{0}(t)\exp\left(-\int_{0}^{t}\gamma_{0}(r)m(r)dr\right)\right]\nonumber\\&
		\hspace*{1.2cm}+\int_{0}^{t}\int_{D}\gamma(t-s)\exp\left(-\int_{s}^{t}\gamma(r-s)m(r)dr\right)\mu(d\gamma)m(s)\overline{\mathfrak{S}}^{(m)}(s)ds\nonumber\\
		&=\E\left[\gamma_{0}(t)\exp\left(-\int_{0}^{t}\gamma_{0}(r)m(r)dr\right)\right]\nonumber\\&
		\hspace*{1.2cm}+\int_{0}^{t}\E\left[\gamma(t-s)\exp\left(-\int_{s}^{t}\gamma(r-s)m(r)dr\right)\right]m(s)\overline{\mathfrak{S}}^{(m)}(s)ds.
	\end{align}
	
	Hence from \eqref{eqn0} and \eqref{eqn4}, $\overline{\mathfrak{F}}^{(m)}=m$ if and only if $(\overline{\mathfrak{S}}^{(m)},m)$ solves \eqref{LLN_xy}. Consequently, by Theorem~\ref{ExistUniq},  there exists a unique element $m^*\in D_+$ such that $\overline{\mathfrak{F}}^{(m^*)}=m^*$ and that $(\overline{\mathfrak{S}}^{(m^*)},m^*)$ solves \eqref{LLN_xy}.
	This concludes the proof of the lemma.
\end{proof}

\medskip

We next consider the sequence $(Q_k)_{k\geq1}$ of Poisson random measures introduced in Section~\ref{sec-model} and for each $k\geq1,$ we define the process $\{A_k(t),\,t\geq0\}$: 
\[A_k(t)=\int_{0}^{t}\int_{0}^{+\infty}\indic{u\leq\Upsilon_k(r^-)}Q_k(dr,du),\] where
\[\Upsilon_k(t)=\gamma_{k,A_k(t)}(\varsigma_{k}(t))\overline{\mathfrak{F}}(t),\] 
and
$\varsigma_k$ is defined in the same manner as $\varsigma^N_1$ with $ A_k$ instead of $ A^N_1$. 
(This definition follows a similar idea to Lemma~\ref{exist_A} as in \cite{chevallier2017mean}.)
In this definition we use the same $(\lambda_{k,i},\gamma_{k,i},Q_k)$ as in the definition of the model in Section~\ref{sec-model}. 
Moreover, since  $\left((\lambda_{k,i})_i,(\gamma_{k,i})_i,Q_k\right)_{k\geq1}$ are i.i.d,  $(A_k)_{k\geq1}$ are also i.i.d.
\begin{remark}\label{rq-6.2}
	From Lemma~\ref{exist_A} we have
	\begin{equation*}
		\overline{\mathfrak F}(t)=\E\left[\lambda_{1,A_1(t)}(\varsigma_{1}(t))\right]\text{ and }	\overline{\mathfrak S}(t)=\E\left[\gamma_{1,A_1(t)}(\varsigma_{1}(t))\right].
	\end{equation*} 
\end{remark}
Now for each $k\geq1$, we compare the process $A^N_k(t),\,t\geq0$ with the process $A_k(t),\,t\geq0$.

\begin{lemma}\label{lem_inq}For $k\in\mathbb{N}$ and $T\geq0$, 
	\begin{equation}
		\mathbb{E}\left[\sup_{t\in[0,T]}\left|A^N_k(t)-A_k(t)\right|\right]\leq\int_{0}^{T}\mathbb{E}\Big[\left|\Upsilon^N_k(t)-\Upsilon_k(t)\right|\Big]dt=:\delta^N(T)\label{eqA}
	\end{equation}
	and 
	\begin{equation*}
		\mathbb{E}\left[\sup_{t\in[0,T]}\left|\varsigma^N_k(t)-\varsigma_k(t)\right|\right]\leq T\delta^N(T).
	\end{equation*}
	Moreover, 
	\begin{equation}\delta^N(T)\leq\frac{\lambda^*}{\sqrt{N}}T\exp(2\lambda^*T).\label{eqdelta}\end{equation}
\end{lemma}
\begin{proof}
	We adapt here the proof of Theorem 4.1 in \cite{chevallier2017mean} to our setting. Since \[\left|A^N_k(t)-A_k(t)\right|\leq\int_{0}^{t}\int_{0}^{+\infty}\indic{\min\left(\Upsilon_k^N(r^-),\Upsilon_k(r^-)\right)< u\leq\max\left(\Upsilon_k^N(r^-),\Upsilon_k(r^-)\right)}Q_k(dr,du),\]
	we have
	\begin{equation*}
		\mathbb{E}\left[\sup_{t\in[0,T]}\left|A^N_k(t)-A_k(t)\right|\right]\leq\int_{0}^{T}\mathbb{E}\Big[\left|\Upsilon^N_k(t)-\Upsilon_k(t)\right|\Big]dt=\delta^N(T).
	\end{equation*}
	We recall that \[\Upsilon^N_k(t)=\gamma_{k,A_k^N(t)}(\varsigma_{k}^N(t))\overline{\mathfrak{F}}^N(t) \quad \text{ and } \quad \Upsilon_k(t)=\gamma_{k,A_k(t)}(\varsigma_{k}(t))\overline{\mathfrak{F}}(t).\]
	However, since $\gamma_{k,i}\leq1$ and  $0\leq\overline{\mathfrak{F}}^N(t),\,\overline{\mathfrak{F}}(t)\leq\lambda^\ast$, we obtain 
	\begin{align}\label{EL0}
		\mathbb{E}\Big[\left|\Upsilon^N_k(t)-\Upsilon_k(t)\right|\Big]&\leq\mathbb{E}\left[\left|\Upsilon^N_k(t)-\Upsilon_k(t)\right|\indic{A^N_k(t)=A_k(t),\varsigma_{k}(t)=\varsigma^N_k(t)}\right]+\nonumber\\&\hspace*{3cm}\lambda^*\mathbb{P}\left(A^N_k(t)\neq A_k(t)\text{ or }\varsigma_{k}(t)\neq \varsigma^{N}_k(t)\right).
	\end{align}
	On the other hand, using $\overline{\mathfrak{F}}(t)=\E\left[\lambda_{1,A_1(t)}(\varsigma_{1}(t))\right]$, we have
	\begin{align}
		&  \mathbb{E}\left[\left|\Upsilon^N_k(t)-\Upsilon_k(t)\right|\indic{A^N_k(t)=A_k(t),\varsigma_{k}(t)=\varsigma^N_k(t)}\right] \nonumber \\
		&\leq\E\left[\left|\overline{\mathfrak F}^N(t)-\overline{\mathfrak{F}}(t)\right|\right]\nonumber\\
		&=\mathbb{E}\left[\left|\frac{1}{N}\sum_{j=1}^{N}\left(\lambda_{j,A^N_j(t)}(\varsigma^N_{j}(t))-\mathbb{E}\left[\lambda_{1,A_1(t)}(\varsigma_{1}(t))\right]\right)\right|\right]\nonumber\\
		&\leq\mathbb{E}\left[\left|\frac{1}{N}\sum_{j=1}^{N}\left(\lambda_{j,A^N_j(t)}(\varsigma^N_{j}(t))-\lambda_{j,A_j(t)}(\varsigma_{j}(t))\right)\right|\right]\nonumber\\
		&\hspace*{1cm}+\mathbb{E}\left[\left|\frac{1}{N}\sum_{j=1}^{N}\left(\lambda_{j,A_j(t)}(\varsigma_{j}(t))-\mathbb{E}\left[\lambda_{1,A_1(t)}(\varsigma_{1}(t))\right]\right)\right|\right]\label{EL1}.
	\end{align}

	Since $((\lambda_{k,i})_i,A_k,\varsigma_k)_k$ are i.i.d, $(\lambda_{k,A_k(t)}(\varsigma_k(t))_k$ are i.i.d., hence, by H\"older's inequality we have 
	\begin{align}\label{EL2}
		& \mathbb{E}\left[\left|\frac{1}{N}\sum_{j=1}^{N}\left(\lambda_{j,A_j(t)}(\varsigma_{j}(t))-\mathbb{E}\left[\lambda_{1,A_1(t)}(\varsigma_{1}(t))\right]\right)\right|\right] \nonumber \\
		&\leq\frac{1}{N}\left(\E\left[\left(\sum_{j=1}^{N}\left(\lambda_{j,A_j(t)}(\varsigma_{j}(t))-\mathbb{E}\left[\lambda_{j,A_j(t)}(\varsigma_{j}(t))\right]\right)\right)^{2}\right]\right)^{\frac{1}{2}}\nonumber\\
		&=\frac{1}{N}\left(\sum_{j=1}^{N}\E\left[\left(\lambda_{j,A_j(t)}(\varsigma_{j}(t))-\mathbb{E}\left[\lambda_{j,A_j(t)}(\varsigma_{j}(t))\right]\right)^2\right]\right)^{\frac{1}{2}}\nonumber\\
		&\leq\frac{\lambda^*}{\sqrt{N}}.
	\end{align}
	Here the equality holds because $(\lambda_{k,A_k(t)}(\varsigma_{k}(t))-\mathbb{E}\left[\lambda_{k,A_k(t)}(\varsigma_{k}(t))\right])_k$ are i.i.d. and the last inequality holds since $\lambda_{k,i}$ is bounded by $\lambda^*$. 
	
	In addition, as $(A_j^N)_j$ are exchangeable we have
	\begin{align}\label{egals}
		& \mathbb{E}\left[\left|\frac{1}{N}\sum_{j=1}^{N}\left(\lambda_{j,A^N_j(t)}(\varsigma^N_{j}(t))-\lambda_{j,A_j(t)}(\varsigma_{j}(t))\right)\right|\right] \nonumber\\
		&=\mathbb{E}\left[\left|\frac{1}{N}\sum_{j=1}^{N}\left(\lambda_{j,A^N_j(t)}(\varsigma^N_{j}(t))-\lambda_{j,A_j(t)}(\varsigma_{j}(t))\right)\indic{A_j(t)\neq A_j^N(t)\text{ or }\varsigma_{j}(t)\neq \varsigma_{j}^N(t)}\right|\right] \nonumber \\
		&\leq\frac{\lambda^*}{N}\sum_{j=1}^{N}\mathbb{P}\left(A_j(t)\neq A_j^N(t)\text{ or }\varsigma_{j}(t)\neq \varsigma_{j}^N(t)\right) \nonumber \\
		&=\lambda^*\mathbb{P}\left(A_k(t)\neq A_k^N(t)\text{ or }\varsigma_{k}(t)\neq \varsigma_{k}^N(t)\right).
	\end{align}
	
	Hence from \eqref{EL1}, \eqref{EL2} and \eqref{egals} we have
	\begin{multline*}
		\mathbb{E}\left[\left|\Upsilon^N_k(t)-\Upsilon_k(t)\right|\indic{A^N_k(t)=A_k(t),\varsigma_{k}(t)=\varsigma^N_k(t)}\right]\\ \leq\frac{\lambda^*}{\sqrt{N}}+\lambda^*\mathbb{P}\left(A_k(t)\neq A_k^N(t)\text{ or }\varsigma_{k}(t)\neq \varsigma_{k}^N(t)\right).
	\end{multline*}
	On the other hand, since \[\left\{A^N_k(t)\neq A_k(t)\text{ or }\varsigma_{k}(t)\neq \varsigma^N_k(t)\right\}\subset\left\{\sup_{r\in[0,t]}|A^N_k(r)-A_k(r)|\geq1\right\},\]
	we have 
	\[\mathbb{P}\left(A_k(t)\neq A_k^N(t)\text{ or }\varsigma_{k}(t)\neq \varsigma_{k}^N(t)\right)\leq\E\left[\sup_{r\in[0,t]}|A^N_k(r)-A_k(r)|\right]\leq\delta^N(t).\]
	Thus, from \eqref{EL0}, we have 
	\begin{equation*}
		\mathbb{E}\Big[\left|\Upsilon^N_k(t)-\Upsilon_k(t)\right|\Big]\leq\frac{\lambda^*}{\sqrt{N}}+2\lambda^*\delta^N(t).
	\end{equation*}
	Combining this with \eqref{eqA}, we deduce that for any $T\geq0,$
	\[\delta^N(T)\leq\frac{\lambda^*}{\sqrt{N}}T+2\lambda^*\int_{0}^{T}\delta^N(t)dt,\]
	hence by Gronwall's lemma, we have
	\begin{equation*}\delta^N(T)\leq\frac{\lambda^*}{\sqrt{N}}T\exp(2\lambda^*T).\end{equation*}
	Moreover, 
	\begin{align*}
		\mathbb{E}\left[\sup_{t\in[0,T]}\left|\varsigma^N_k(t)-\varsigma_k(t)\right|\right]&=\mathbb{E}\left[\indic{\{\exists t\in[0,T],\varsigma^N_k(t)\neq\varsigma_k(t)\}}\sup_{t\in[0,T]}\left|\varsigma^N_k(t)-\varsigma_k(t)\right|\right]\\
		&\leq T\mathbb{P}\left(\exists t\in[0,T],\varsigma^N_k(t)\neq\varsigma_k(t)\right)\\
		&=T\mathbb{P}\left(\sup_{t\in[0,T]}\left|A^N_k(t)-A_k(t)\right|\neq0\right)\\
		&\leq T\E\left[\sup_{t\in[0,T]}\left|A^N_k(t)-A_k(t)\right|\right]\\
		&\leq T\delta^N(T).
	\end{align*}
	This concludes the proof of the lemma.
\end{proof}

\medskip
From the proof of Lemma~\ref{lem_inq}, we deduce the following Remark.
\begin{remark}
	For $k\in\mathbb{N}$ and $t\geq0$ we have
	\begin{align*}
		\E\left[\left|\overline{\mathfrak F}^N(t)-\overline{\mathfrak{F}}(t)\right|\right] &\leq \frac{\lambda^*}{\sqrt{N}}\left(1+t\exp(2\lambda^*t)\right), \\
		\E\left[\left|\overline{\mathfrak S}^N(t)-\overline{\mathfrak{S}}(t)\right|\right] &\leq \frac{1}{\sqrt{N}}\left(1+\lambda^*t\exp(2\lambda^*t)\right),
	\end{align*}
	and
	\begin{equation*}
		\E\left[\left|\Upsilon^N_k(t)-\Upsilon_k(t)\right|\right]\leq \frac{\lambda^*}{\sqrt{N}}\left(1+2\lambda^*t\exp(2\lambda^*t)\right).
	\end{equation*}
\end{remark}
From Lemma~\ref{lem_inq}, we deduce the following Lemma.
\begin{lemma}
	For $k\in\mathbb{N}$ and $T\geq0$ we have 
	\begin{equation}
		\E\left[\sup_{t\in[0,T]}\left|\gamma_{k,A_k^N(t)}(\varsigma^N_k(t))-\gamma_{k,A_k(t)}(\varsigma_k(t))\right|\right]\leq\frac{\lambda^*}{\sqrt{N}}T\exp(2\lambda^*T),\label{eqgam}
	\end{equation}
	\begin{equation}
		\E\left[\sup_{t\in[0,T]}\left|\lambda_{k,A_k^N(t)}(\varsigma^N_k(t))-\lambda_{k,A_k(t)}(\varsigma_k(t))\right|\right]\leq\frac{\lambda^{*2}}{\sqrt{N}}T\exp(2\lambda^*T),\label{eqlam}
	\end{equation}
	and
	\begin{equation}
		\E\left[\sup_{t\in[0,T]}\left|\indic{\varsigma^N_k(t)<\eta_{k,A_k^N(t)}}-\indic{\varsigma_k(t)<\eta_{k,A_k(t)}}\right|\right]\leq\frac{\lambda^*}{\sqrt{N}}T\exp(2\lambda^*T).\label{eqind}
	\end{equation}
\end{lemma}
\begin{proof}
	From \eqref{eqA} and the fact that $\gamma_{k,i}\le1$, we have
	\begin{multline*}
		\E\left[\sup_{t\in[0,T]}\left|\gamma_{k,A_k^N(t)}(\varsigma^N_k(t))-\gamma_{k,A_k(t)}(\varsigma_k(t))\right|\right]\\
		\begin{aligned}
			&=\E\left[\sup_{t\in[0,T]}\left|\gamma_{k,A_k^N(t)}(\varsigma^N_k(t))-\gamma_{k,A_k(t)}(\varsigma_k(t))\right|\indic{\sup_{t\in[0,T]}\left|A_k^N(t)-A_k(t)\right|\geq1}\right]\\
			&\leq\mathbb{P}\left(\sup_{t\in[0,T]}\left|A^N_k(t)-A_k(t)\right|\geq1\right)\\
			&\leq\delta^N(T).
		\end{aligned}
	\end{multline*}
	Similarly we also have 
	\begin{numcases}{}
		\E\left[\sup_{t\in[0,T]}\left|\lambda_{k,A_k^N(t)}(\varsigma^N_k(t))-\lambda_{k,A_k(t)}(\varsigma_k(t))\right|\right]\leq\lambda^*\delta^N(T),\nonumber\\
		\E\left[\sup_{t\in[0,T]}\left|\indic{\varsigma^N_k(t)<\eta_{k,A_k^N(t)}}-\indic{\varsigma_k(t)<\eta_{k,A_k(t)}}\right|\right]\leq\delta^N(T)\nonumber.
	\end{numcases}
	Hence the claims follow from \eqref{eqdelta}. 
\end{proof}

\begin{proof}[Completing the proof of Theorem~\ref{thm-FLLN}]
	For $t\geq0$, we have
	\begin{align*}
		\overline{\mathfrak{F}}^N(t)&=\frac{1}{N}\sum_{k=1}^{N}\lambda_{k,A^N_k(t)}(\varsigma^{N}_k(t))\\
		&=\frac{1}{N}\sum_{k=1}^{N}\left(\lambda_{k,A^N_k(t)}(\varsigma^{N}_k(t))-\lambda_{k,A_k(t)}(\varsigma_k(t))\right)+\frac{1}{N}\sum_{k=1}^{N}\lambda_{k,A_k(t)}(\varsigma_k(t)),
	\end{align*}
	and 
	\begin{align*}
		\overline{\mathfrak{S}}^N(t)&=\frac{1}{N}\sum_{k=1}^{N}\gamma_{k,A^N_k(t)}(\varsigma^{N}_k(t))\\
		&=\frac{1}{N}\sum_{k=1}^{N}\left(\gamma_{k,A^N_k(t)}(\varsigma^{N}_k(t))-\gamma_{k,A_k(t)}(\varsigma_k(t))\right)+\frac{1}{N}\sum_{k=1}^{N}\gamma_{k,A_k(t)}(\varsigma_k(t)).
	\end{align*}
	From \eqref{eqgam} and \eqref{eqlam}, we have 
	\begin{equation*}\left\{
		\begin{aligned}
			\E\left[\frac{1}{N}\sum_{k=1}^{N}\sup_{t \in [0,T]}\left|\lambda_{k,A^N_k(t)}(\varsigma^{N}_k(t))-\lambda_{k,A_k(t)}(\varsigma_k(t))\right|\right]\leq\frac{\lambda^{*2}}{\sqrt{N}}T\exp(2\lambda^*T),\\
			\E\left[\frac{1}{N}\sum_{k=1}^{N}\sup_{t \in [0,T]}\left|\gamma_{k,A^N_k(t)}(\varsigma^{N}_k(t))-\gamma_{k,A_k(t)}(\varsigma_k(t))\right|\right]\leq\frac{\lambda^{*}}{\sqrt{N}}T\exp(2\lambda^*T).
		\end{aligned}\right.
	\end{equation*}
	Hence, 
	\begin{multline*}
		\left(\frac{1}{N}\sum_{k=1}^{N}\left(\gamma_{k,A^N_k(t)}(\varsigma^{N}_k(t))-\gamma_{k,A_k(t)}(\varsigma_k(t))\right), \right. \\ \left. \frac{1}{N}\sum_{k=1}^{N}\left(\lambda_{k,A^N_k(t)}(\varsigma^{N}_k(t))-\lambda_{k,A_k(t)}(\varsigma_k(t))\right)\right) \xrightarrow[N\to+\infty]{}(0,0)
	\end{multline*}
	locally uniformly in $t$. 
	
	Moreover, as $\left(\gamma_{k,A_k(\cdot)}(\varsigma_k(\cdot)),\lambda_{k,A_k(\cdot)}(\varsigma_k(\cdot))\right)_k$ is a collection of i.i.d. random variables in $D^2$, by the law of large numbers in $D^2$ \cite[Theorem~$1$]{rao1963law},
	\begin{multline*}\left(\frac{1}{N}\sum_{k=1}^{N}\gamma_{k,A_k(\cdot)}(\varsigma_k(\cdot)),\frac{1}{N}\sum_{k=1}^{N}\lambda_{k,A_k(\cdot)}(\varsigma_k(\cdot))\right)\\\xrightarrow[N\to+\infty]{\mathbb{P}}\left(\E\left[\gamma_{1,A_1(\cdot)}(\varsigma_{1}(\cdot))\right],\E\left[\lambda_{1,A_1(\cdot)}(\varsigma_{1}(\cdot))\right]\right)\text{ in }D^2.\end{multline*}
	
	We have shown in the proof on Lemma~\ref{exist_A} that the pair $(\overline{\mathfrak{S}},\overline{\mathfrak{F}})$ given in Remark~\ref{rq-6.2} solves the set of equations \eqref{LLN_xy}. This proves the convergence \eqref{eqn-mfk-SF-conv}. 
	
	\medskip
	
	For $(\bar{U}^N, \bar{I}^N)$, we have for $t\geq0$,
	\begin{align*}
		\overline{I}^N(t)&=\frac{1}{N}\sum_{k=1}^{N}\indic{\varsigma^N_k(t)<\eta_{k,A_k^N(t)}}\\
		&=\frac{1}{N}\sum_{k=1}^{N}\left(\indic{\varsigma^N_k(t)<\eta_{k,A_k^N(t)}}-\indic{\varsigma_k(t)<\eta_{k,A_k(t)}}\right)+\frac{1}{N}\sum_{k=1}^{N}\indic{\varsigma_k(t)<\eta_{k,A_k(t)}},
	\end{align*}
	and 
	\begin{align*}
		\overline{U}^N(t)&=\frac{1}{N}\sum_{k=1}^{N}\indic{\varsigma^N_k(t)\geq\eta_{k,A_k^N(t)}}\\
		&=\frac{1}{N}\sum_{k=1}^{N}\left(\indic{\varsigma^N_k(t)\geq\eta_{k,A_k^N(t)}}-\indic{\varsigma_k(t)\geq\eta_{k,A_k(t)}}\right)+\frac{1}{N}\sum_{k=1}^{N}\indic{\varsigma_k(t)\geq\eta_{k,A_k(t)}}.
	\end{align*}
	As above, from \eqref{eqind}, we deduce that 
	\begin{multline*}\left(\frac{1}{N}\sum_{k=1}^{N}\left(\indic{\varsigma^N_k(t)\geq\eta_{k,A_k^N(t)}}-\indic{\varsigma_k(t)\geq\eta_{k,A_k(t)}}\right),\frac{1}{N}\sum_{k=1}^{N}\left(\indic{\varsigma^N_k(t)<\eta_{k,A_k^N(t)}}-\indic{\varsigma_k(t)<\eta_{k,A_k(t)}}\right)\right)\\\xrightarrow[N\to+\infty]{}(0,0)\end{multline*}locally uniformly in $t$. 
	
	Moreover, as $\left(\varsigma_k(t),\eta_{k,A_k(t)}\right)_k$ is a collection of i.i.d. random variables in $D^2$, by the law of large numbers in $D^2$ \cite[Theorem~$1$]{rao1963law},
	\begin{multline*}\left(\frac{1}{N}\sum_{k=1}^{N}\indic{\varsigma_k(t)\geq\eta_{k,A_k(t)}},\frac{1}{N}\sum_{k=1}^{N}\indic{\varsigma_k(t)<\eta_{k,A_k(t)}}\right)\\\xrightarrow[N\to+\infty]{\mathbb{P}}\left(\E\left[\indic{\varsigma_1(t)\geq\eta_{1,A_1(t)}}\right],\E\left[\indic{\varsigma_1(t)<\eta_{1,A_1(t)}}\right]\right)\text{ in }D^2.\end{multline*}
	Recall the definition of $ \overline{I}(t) $ and $ \overline{U}(t) $ in \eqref{eqn-barS} and \eqref{eqn-barI}. In order to complete the proof of Theorem~\ref{thm-FLLN}, it remains to verify that: 
	\[\overline{I}(t)=\E\left[\indic{\varsigma_1(t)<\eta_{1,A_1(t)}}\right] \text{ and }\overline{U}(t)=\E\left[\indic{\varsigma_1(t)\geq\eta_{1,A_1(t)}}\right].\]
	Denote the jump times of the process $A_1$ by $(\tau_i)_{i\geq1}$,  if $t-\tau_i<\eta_{1,i},\, \lambda_{1,i}(t-\tau_i)\neq0$.  From Assumption~\ref{AS-lambda}, we deduce that
	$$\int_{\tau_i}^{t}\gamma_{1,i}(r-\tau_i)\overline{\mathfrak F}(r)dr=0.$$
	So $A_1(t)=A_1(\tau_{i})$. 
	Since $\indic{t<\eta_{1,0}}=\indic{t<\eta_{1,0}}\indic{A_1(t)=0}$ a.s., we have 
	\begin{align*}
		\E\left[\indic{\varsigma_1(t)<\eta_{1,A_1(t)}}\right]&=\E\left[\indic{t<\eta_{1,0}}\indic{A_1(t)=0}+\sum_{i\geq1}\indic{t-\tau_i<\eta_{1,i}}\indic{A_1(t)=i}\right]\nonumber\\
		&=\P\left(\eta_{1,0}>t\big|\eta_{1,0}>0\right)\P\left(\eta_{1,0}>0\right)+\E\left[\sum_{i\geq1}\indic{t-\tau_i<\eta_{1,i}}\indic{\tau_i\leq t}\right]\nonumber\\
		&=\overline{I}(0)F^c_0(t)+\sum_{i\geq1}\E\left[\P\left(t-\tau_i<\eta_{1,i}\big|\tau_i\right)\indic{\tau_i\leq t}\right].
	\end{align*}
	Moreover, since $\eta_{1,i}$ and $\tau_i$ are independent, we obtain 
	\begin{align}
		\E\left[\indic{\varsigma_1(t)<\eta_{1,A_1(t)}}\right]&=\overline{I}(0)F^c_0(t)+\sum_{i\geq1}\E\left[F^c\left(t-\tau_i\right)\indic{\tau_i\leq t}\right]\nonumber\\
		&=\overline{I}(0)F^c_0(t)+\E\left[\int_{0}^{t}F^c(t-s)A_1(ds)\right]\nonumber\\
		&=\overline{I}(0)F^c_0(t)+\int_{0}^{t}F^c(t-s)\overline{\mathfrak F}(s)\overline{\mathfrak S}(s)ds.\label{eqd1}
	\end{align}
	If $\eta_{1,0}>t,\,\lambda_{1,0}(t)>0\,(\eta_{1,1}>t,\,\lambda_{1,1}(t)>0)$, from Assumption \ref{AS-lambda} by using the fact that $\lambda_{1,1}(t)>0 \,(\,\lambda_{1,0}(t)>0)$ implies that $\gamma_{1,1}(s)=0\, (\,\gamma_{1,0}(s)=0)$ for $0\leq s\leq t$, from \eqref{eqd1} we also have, 
	\begin{align*}
		\E\left[\indic{\varsigma_1(t)<\eta_{1,A_1(t)}}\right]&=\E\left[\indic{\eta_{1,0}>t}\exp\left(-\int_{0}^{t}\gamma_{0}(r)\overline{\mathfrak F}(r)dr\right)\right]\\&\hspace{1.5cm}+\int_{0}^{t}\E\left[\indic{\eta_{1,1}>t-s}\exp\left(-\int_{s}^{t}\gamma(r-s)\overline{\mathfrak F}(r)dr\right)\right]\overline{\mathfrak F}(s)\overline{\mathfrak S}(s)ds,
	\end{align*}
	which combined with \eqref{eqG1} yields
	\begin{align*}
		\E\left[\indic{\varsigma_1(t)\geq\eta_{1,A_1(t)}}\right]&=1-\E\left[\indic{\varsigma_1(t)<\eta_{1,A_1(t)}}\right]\\
		&=\E\left[\indic{t\geq\eta_{1,0}}\exp\left(-\int_{0}^{t}\gamma_{0}(r)\overline{\mathfrak F}(r)dr\right)\right]\\&\hspace{1.5cm}+\int_{0}^{t}\E\left[\indic{t-s\geq\eta_{1,1}}\exp\left(-\int_{s}^{t}\gamma(r-s)\overline{\mathfrak F}(r)dr\right)\right]\overline{\mathfrak F}(s)\overline{\mathfrak S}(s)ds.
	\end{align*}
	This completes the proof of Theorem~\ref{thm-FLLN}. 
\end{proof}

\medskip

We also have the following convergence result on the empirical measure of the processes $(A^N_k(t),\varsigma^N_k(t))_{t\geq0}$.
It is not used in our analysis, but a worth to be established. 

\begin{theorem}\label{LLN-P}
	\begin{equation}  \label{eqn-LLN-P}
		\frac{1}{N}\sum_{k=1}^{N}\delta_{(A^N_k(t),\varsigma^N_k(t))_{t\geq0}}\xrightarrow[N\to+\infty]{\mathbb{P}}\mathcal{L}\left((A_1(t),\varsigma_1(t))_{t\geq0}\right)\text{ in }\mathcal{P}(D^2).
	\end{equation}
\end{theorem}
\begin{proof}
	By Lemma~\ref{lem_inq}, we have
	\begin{multline*}
		\mathbb{E}\left[\sup_{t\in[0,T]}\Big(\left|A^N_1(t)-A_1(t)\right|+\left|A^N_2(t)-A_2(t)\right|+\left|\varsigma^N_1(t)-\varsigma_1(t)\right|+\left|\varsigma^N_2(t)-\varsigma_2(t)\right|\Big)\right]\\\xrightarrow[N\to+\infty]{}0.
	\end{multline*}
	Hence 
	\begin{multline*}\left(\left\{(A^N_1(t),\varsigma^N_1(t))\right\}_{t\geq0},\left\{(A^N_2(t),\varsigma^N_2(t))\right\}_{t\geq0}\right)\\\xrightarrow[N\to+\infty]{\mathbb{P}}
		\left(\left\{(A_1(t),\varsigma_1(t))\right\}_{t\geq0},\left\{(A_2(t),\varsigma_2(t))\right\}_{t\geq0}\right)\end{multline*}
	in $D^2\times D^2$. 
	Thus since the processes are exchangeable and $D$ is a separable metric space, by \cite[Proposition~$2.2$ $i)$, page $177$]{sznitman1991topics}, the convergence in \eqref{eqn-LLN-P} holds. 
\end{proof}

\bigskip

\section{Proofs for the endemic equilibrium} \label{sec-proofs-endemic}

In this section, we prove the results on the endemic equilibrium behaviors.  We proceed in two subsections to prove the results in the scenarios $R_0<\E\left[\frac{1}{\gamma_*}\right]$ and $R_0\geq\E\left[\frac{1}{\gamma_*}\right]$. We have a complete theory in the first scenario as stated in Theorem~\ref{thm-eqlm1}, which we prove first. We then establish some of the partial results in the second scenario.

\subsection{Proof of Theorem~\ref{thm-eqlm1}}

\begin{proof}[Proof of Theorem~\ref{thm-eqlm1}]
	This theorem is proved in two cases: $\mathbb{P}\left(\gamma_\ast=0\right)>0$ and $\mathbb{P}\left(\gamma_\ast=0\right)=0$.

	{\it Case 1: $\mathbb{P}\left(\gamma_\ast=0\right)>0$.} 	
	Recall \eqref{eqG1}.	Note that, since $ \gamma_* = 0 $ implies $ \gamma(t) = 0 $ for all $ t $,
	\begin{align*}
		\E \left[ \exp\left( - \int_{s}^{t} \gamma(r-s) \overline{\mathfrak F}(r) dr \right) \right] \geq \P (\gamma_* = 0).
	\end{align*}
	As a result, from \eqref{eqG1}, for all $ t \geq 0 $,
	\begin{align} \label{bound_new_infections}
		\int_{0}^{t} \overline{\mathfrak S}(s) \overline{\mathfrak F}(s) ds \leq \frac{1}{\P (\gamma_* = 0)}.
	\end{align}
	Consequently, 
	\begin{equation}\label{eqn1}
		\int_{0}^{+\infty}\overline{\mathfrak S}(s) \overline{\mathfrak F}(s) ds<+\infty.
	\end{equation}
	
	Recall that, by Assumption~\ref{hyp-ga-1}, $\overline{\lambda}_0(t)\to0,\,\overline{\lambda}(t)\to0$ as $t\to+\infty$ (see \eqref{lim_lambda_bar}). 
	By the dominated convergence theorem  applied to the second line of \eqref{LLN_GF}, using \eqref{eqn1}, we obtain 
	\begin{equation*}
		\lim_{t\to+\infty}\overline{\mathfrak F}(t)=\lim_{t\to+\infty}\int_{0}^{t}\overline{\lambda}(t-s)\overline{\mathfrak{S}}(s)\overline{\mathfrak{F}}(s)ds=0.
	\end{equation*}  
	Next by the dominated convergence theorem applied to the first line in \eqref{LLN_GF} using \eqref{eqn1} again,
	$\overline{\mathfrak S}_\ast=\lim_{t\to+\infty}\overline{\mathfrak S}(t)$ as given in \eqref{eqn-mfkS*-1}. 
	This concludes the proof of the first case of Theorem~\ref{conv1-f}.
	
	\medskip 
	
	{\it Case 2: $\mathbb{P}\left(\gamma_\ast=0\right)=0$.}
	We first note that
	\begin{equation} \label{eq_equil1-a}
		\int_{0}^{+\infty}\overline{\mathfrak{F}}(u)du<+\infty\Leftrightarrow \int_{0}^{+\infty}\overline{\mathfrak S}(s)\overline{\mathfrak{F}}(s)ds<+\infty.
	\end{equation}
	Indeed, from the second line in \eqref{LLN_GF},  
	and from Fubuni's theorem
	\begin{equation}\label{eq_equil1}
		\int_{0}^{+\infty}\overline{\mathfrak{F}}(u)du=\overline{I}(0)\int_{0}^{+\infty}\overline{\lambda}_0(u)du+R_0\int_{0}^{+\infty}\overline{\mathfrak S}(s)\overline{\mathfrak{F}}(s)ds.
	\end{equation}
	Next from Assumption~\ref{hyp-ga-1}.(iii), we have
	\begin{equation} \label{eq_equil1-b}
		\int_{0}^{+\infty}\overline{\lambda}_0(u)du\leq R_0<+\infty.
	\end{equation}
	This combined with \eqref{eq_equil1} implies \eqref{eq_equil1-a}. 
	Thus from the proof of Case 1 and \eqref{eq_equil1-a}, it suffices to show that  $\int_{0}^{+\infty}\overline{\mathfrak{F}}(s)ds<+\infty$.
	We prove this claim by contradiction. Suppose that 
	\begin{equation}
		\int_{0}^{+\infty}\overline{\mathfrak{F}}(u)du=+\infty.\label{eq_lim_F}
	\end{equation}
	By the second line in \eqref{LLN_GF}, using Fubuni's theorem, we obtain 
	\begin{align*}\label{eq_IF}
		\int_{0}^{t}\overline{\mathfrak{F}}(u)du&=\overline{I}(0)\int_{0}^{t}\overline{\lambda}_0(u)du+\int_{0}^{t}\left(\int_{0}^{t-u}\overline{\lambda}(s)ds\right)\overline{\mathfrak{S}}(u)\overline{\mathfrak F}(u)du\\
		&=\overline{I}(0)\int_{0}^{t}\overline{\lambda}_0(u)du+R_0\int_{0}^{t}\overline{\mathfrak{S}}(u)\overline{\mathfrak F}(u)du-\int_{0}^{t}\left(\int_{t-u}^{+\infty}\overline{\lambda}(s)ds\right)\overline{\mathfrak{S}}(u)\overline{\mathfrak F}(u)du.
	\end{align*}
	Consequently, 
	\begin{equation}\label{frac_lim}
		\frac{\int_{0}^{t}\overline{\mathfrak{S}}(u)\overline{\mathfrak F}(u)du}{\int_{0}^{t}\overline{\mathfrak{F}}(u)du}=\frac{1}{R_0}+\frac{\int_{0}^{t}\left(\int_{t-u}^{+\infty}\overline{\lambda}(s)ds\right)\overline{\mathfrak{S}}(u)\overline{\mathfrak F}(u)du}{R_0\int_{0}^{t}\overline{\mathfrak{F}}(u)du}-\frac{\overline{I}(0)\int_{0}^{t}\overline{\lambda}_0(u)du}{R_0\int_{0}^{t}\overline{\mathfrak{F}}(u)du}.
	\end{equation}
	By \eqref{eq_equil1-b} and \eqref{eq_lim_F}, we have
	\begin{equation}\label{eq_F_1}
		\frac{\int_{0}^{t}\overline{\lambda}_0(u)du}{\int_{0}^{t}\overline{\mathfrak{F}}(u)du}\xrightarrow[t\to+\infty]{}0.
	\end{equation}
	In addition, since $\int_{t}^{+\infty}\overline{\lambda}(s)ds\to0$ as $t\to+\infty$,  for $\epsilon>0$ there exists $T_\epsilon>0$ such that $\int_{T_\epsilon}^{+\infty}\overline{\lambda}(s)ds < \epsilon$. Hence, for $t\geq T_\epsilon$,
	\begin{align*}
		& \int_{0}^{t}\left(\int_{t-u}^{+\infty}\overline{\lambda}(s)ds\right)\overline{\mathfrak{S}}(u)\overline{\mathfrak F}(u)du \\
		&=\int_{0}^{t}\left(\int_{u}^{+\infty}\overline{\lambda}(s)ds\right)\overline{\mathfrak{S}}(t-u)\overline{\mathfrak F}(t-u)du\\
		&\leq\int_{0}^{T_\epsilon}\left(\int_{u}^{+\infty}\overline{\lambda}(s)ds\right)\overline{\mathfrak{S}}(t-u)\overline{\mathfrak F}(t-u)du+\epsilon\int_{T_\epsilon}^{t}\overline{\mathfrak{S}}(t-u)\overline{\mathfrak F}(t-u)du\\
		&\leq R_0\lambda_* T_\epsilon+\epsilon\int_{0}^{t}\overline{\mathfrak F}(u)du.
	\end{align*}
	Thus by \eqref{eq_lim_F}, we have 
	\begin{equation}\label{eq_F_2}
		\frac{\int_{0}^{t}\left(\int_{t-u}^{+\infty}\overline{\lambda}(s)ds\right)\overline{\mathfrak{S}}(u)\overline{\mathfrak F}(u)du}{\int_{0}^{t}\overline{\mathfrak{F}}(u)du}\xrightarrow[t\to+\infty]{}0.
	\end{equation}
	Hence under the assumption \eqref{eq_lim_F}, from \eqref{frac_lim}, \eqref{eq_F_1} and \eqref{eq_F_2}, we obtain 
	\begin{equation}
		\frac{\int_{0}^{t}\overline{\mathfrak{S}}(u)\overline{\mathfrak F}(u)du}{\int_{0}^{t}\overline{\mathfrak{F}}(u)du}\xrightarrow[t\to+\infty]{}\frac{1}{R_0}.\label{eq_frac_lim}
	\end{equation}
	On the other hand, from \eqref{eqG1} and the fact that $\gamma\leq\gamma_*$ we have
	\begin{equation}\label{eq_in_G}
		\int_{0}^{u} \E \left[ \exp \left( - \gamma_\ast\int_{s}^{u}\overline{\mathfrak F}(r) dr \right) \right] \overline{\mathfrak S}(s) \overline{\mathfrak F}(s) ds<1.
	\end{equation}
	Next, multiplying by $\overline{\mathfrak F}(u)$ and integrating from $0$ to $t$ both sides of \eqref{eq_in_G},   we have
	\begin{equation*}
		\int_{0}^{t}\left(\int_{0}^{u} \E \left[ \overline{\mathfrak F}(u)\exp \left( - \gamma_\ast\int_{s}^{u}\overline{\mathfrak F}(r) dr \right) \right] \overline{\mathfrak S}(s) \overline{\mathfrak F}(s) ds\right)du<\int_{0}^{t}\overline{\mathfrak F}(u)du,
	\end{equation*}
	and by Fubuni's theorem,
	\begin{equation*}
		\int_{0}^{t}\E \left[ \int_{s}^{t}\overline{\mathfrak F}(u)\exp \left( - \gamma_\ast\int_{s}^{u}\overline{\mathfrak F}(r) dr \right) du\right]\overline{\mathfrak S}(s) \overline{\mathfrak F}(s)ds<\int_{0}^{t}\overline{\mathfrak F}(u)du,
	\end{equation*}
	from which we obtain 
	\begin{equation*}
		\int_{0}^{t}\E \left[\frac{1-\exp \left( - \gamma_\ast\int_{s}^{t}\overline{\mathfrak F}(r) dr \right)}{\gamma_*}\right]\overline{\mathfrak S}(s) \overline{\mathfrak F}(s)ds<\int_{0}^{t}\overline{\mathfrak F}(u)du.
	\end{equation*}
	Hence for $0<\epsilon\leq1$, since $\indic{\gamma_\ast\geq\epsilon}\leq1$,
	\begin{equation*}
		\int_{0}^{t}\E \left[\frac{1-\exp \left( - \gamma_\ast\int_{s}^{t}\overline{\mathfrak F}(r) dr \right)}{\gamma_*}\indic{\gamma_\ast\geq\epsilon}\right]\overline{\mathfrak S}(s) \overline{\mathfrak F}(s)ds<\int_{0}^{t}\overline{\mathfrak F}(u)du.
	\end{equation*}
	Thus,
	\begin{multline}\label{eq_d}
		\E \left[\frac{1}{\gamma_*}\indic{\gamma_\ast\geq\epsilon}\right]\int_{0}^{t}\overline{\mathfrak S}(s) \overline{\mathfrak F}(s)ds<\int_{0}^{t}\overline{\mathfrak F}(u)du\\+\int_{0}^{t}\E \left[\frac{1}{\gamma_*}\indic{\gamma_\ast\geq\epsilon}\exp \left( - \gamma_\ast\int_{s}^{t}\overline{\mathfrak F}(r) dr \right)\right]\overline{\mathfrak S}(s) \overline{\mathfrak F}(s)ds.
	\end{multline}
	Moreover, from \eqref{eq_in_G},
	\begin{equation*}
		\int_{0}^{t}\E \left[\frac{1}{\gamma_*}\indic{\gamma_\ast\geq\epsilon}\exp \left( - \gamma_\ast\int_{s}^{t}\overline{\mathfrak F}(r) dr \right)\right]\overline{\mathfrak S}(s) \overline{\mathfrak F}(s)ds\leq\frac{1}{\epsilon}. 
	\end{equation*}
	Consequently, under the assumption \eqref{eq_lim_F}, we have
	\begin{equation*}
		\frac{\int_{0}^{t}\E \left[\frac{1}{\gamma_*}\indic{\gamma_\ast\geq\epsilon}\exp \left( - \gamma_\ast\int_{s}^{t}\overline{\mathfrak F}(r) dr \right)\right]\overline{\mathfrak S}(s) \overline{\mathfrak F}(s)ds}{\int_{0}^{t}\overline{\mathfrak F}(u)du}\xrightarrow[t\to+\infty]{}0.
	\end{equation*}
	This implies that, by \eqref{eq_d}, for all $0<\epsilon\leq1$, 
	\begin{equation*}
		\limsup_{t\to+\infty}\frac{\int_{0}^{t}\overline{\mathfrak{S}}(u)\overline{\mathfrak F}(u)du}{\int_{0}^{t}\overline{\mathfrak{F}}(u)du}\leq\left(\E\left[\frac{1}{\gamma_\ast}\indic{\gamma_\ast\geq\epsilon}\right]\right)^{-1}. 
	\end{equation*}
	Since  $\mathbb{P}\left(\gamma_\ast=0\right)=0,$ we deduce by the monotone convergence theorem that
	\begin{equation*}
		\limsup_{t\to+\infty}\frac{\int_{0}^{t}\overline{\mathfrak{S}}(u)\overline{\mathfrak F}(u)du}{\int_{0}^{t}\overline{\mathfrak{F}}(u)du}\leq\left(\E\left[\frac{1}{\gamma_\ast}\right]\right)^{-1},
	\end{equation*}
	However, this contradicts \eqref{eq_frac_lim} since $R_0<\E\left[\frac{1}{\gamma_\ast}\right]$ by the assumption of Theorem~\ref{thm-eqlm1}. 
		
	This completes the proof of the second case. 
\end{proof}

\subsection[Proofs in the endemic case]{Proofs in the case $R_0\geq\E\left[\frac{1}{\gamma_*}\right]$}
In this subsection, we prove Theorem \ref{conv1-f}.

\begin{proof}[Proof of Theorem \ref{conv1-f}~$(i)$]  
	As $ \overline{\mathfrak S}(t) \to \overline{\mathfrak S}_* $ and $ \overline{\mathfrak F}(t) \to \overline{\mathfrak F}_* $ as $ t \to \infty, R_0<+\infty$ and $\overline{\mathfrak F}(t)\overline{\mathfrak S}(t)\leq\lambda_\ast$, by the dominated convergence theorem, we have
	\begin{equation}\lim_{t\to+\infty}\int_{0}^{t}\overline{\lambda}(t-s)\overline{\mathfrak{S}}(s)\overline{\mathfrak{F}}(s)ds=\lim_{t\to+\infty}\int_{0}^{+\infty}\overline{\lambda}(s)\overline{\mathfrak{S}}(t-s)\overline{\mathfrak{F}}(t-s)\indic{[0,t]}(s)ds=\overline{\mathfrak{S}}_\ast\overline{\mathfrak{F}}_\ast R_0.\label{eqs19}\end{equation}
	Thus by \eqref{eqs19} and the second line of \eqref{LLN_GF}, using the fact that $ \overline{\lambda}_0(t) \to 0 $ as $ t \to \infty $,
	\[\overline{\mathfrak{F}}_\ast=\overline{\mathfrak{S}}_\ast\overline{\mathfrak{F}}_\ast R_0.\]
	As a result, either $ \overline{\mathfrak F}_* = 0 $ or else $ \overline{\mathfrak S}_* = \frac{1}{R_0} $.
	
	In the following, we assume that $ \overline{\mathfrak F}_* > 0 $, then $ \overline{\mathfrak S}_* = \frac{1}{R_0} $.
	
	Recall  \eqref{eqG1}. 
	Since $ \E \left[ \frac{1}{\gamma_*} \right] \leq R_0< + \infty $ and $\mathbb{P}(\gamma_\ast=0)=0,\,\gamma(t+\xi) \overline{\mathfrak{F}}(t) \to \gamma_\ast \overline{\mathfrak{F}}_\ast > 0 $ with probability one when $t\to+\infty$ and hence 
	\[\lim_{t\to+\infty}\int_{0}^{t}\gamma(r+\xi)\overline{\mathfrak{F}}(r)dr=+\infty\text{ and }\lim_{t\to+\infty}\int_{0}^{t}\overline{\mathfrak{F}}(r)dr=+\infty\text{ almost surely.} \]
	It follows that
	\[\lim_{t\to+\infty}\E\left[  \exp\left(-\int_0^t \gamma_0(r) \overline{\mathfrak{F}}(r) dr \right)\right]= 0. \]
	Hence from \eqref{eqG1} we have
	\begin{equation}\label{LLNG1}
		\lim_{t\to+\infty}\int_{0}^{t}\E\left[\exp\left(- \int_s^t \gamma(r-s) \overline{\mathfrak{F}}(r) dr \right)\right]\overline{\mathfrak{S}}(s) \overline{\mathfrak{F}}(s)ds=1.
	\end{equation}
	Since $\overline{\mathfrak{F}}(t)\to\overline{\mathfrak F}_\ast$ when $t\to+\infty$, there exists $t_0>0$ such that for all $t\geq t_0,\,\overline{\mathfrak F}(t)\geq\frac{\overline{\mathfrak F}_\ast}{2}$. 
	Then,
	\begin{multline} \label{integral_split}
		\int_{0}^{t}\E\left[\exp\left(- \int_s^t \gamma(r-s) \overline{\mathfrak{F}}(r) dr \right)\right]\overline{\mathfrak{S}}(s) \overline{\mathfrak{F}}(s)ds \\
		= \int_{0}^{t_0} \E\left[\exp\left(- \int_0^{t-s} \gamma(r) \overline{\mathfrak{F}}(s+r) dr \right)\right]\overline{\mathfrak{S}}(s) \overline{\mathfrak{F}}(s)ds \\
		+ \int_{0}^{t-t_0} \E\left[\exp\left(- \int_0^s \gamma(r) \overline{\mathfrak{F}}(t+r-s) dr \right)\right]\overline{\mathfrak{S}}(t-s) \overline{\mathfrak{F}}(t-s)ds.
	\end{multline}
	The first term on the right hand side converges to zero as $ t \to \infty $ by the dominated convergence theorem.
	%
	%
	On the other hand, for $0<s<t-t_0$, 
	\begin{equation} 
		\exp\left(- \int_{0}^s \gamma(r) \overline{\mathfrak{F}}(r+t-s) dr \right)\leq\exp\left(-\frac{\overline{\mathfrak F}_\ast}{2}\int_{0}^s \gamma(r)dr\right),\label{end1}
	\end{equation}
	and by Assumption~\ref{hyp-ga}, we deduce that  
	\begin{align}
		\int_{0}^{+\infty}\E\left[\exp\left(-\frac{\overline{\mathfrak F}_\ast}{2}\int_{0}^s \gamma(r)dr\right) \right] ds &\leq\E\left[t_\ast\right]+\E\left[\int_{0}^{+\infty}\exp\left(-\frac{\overline{\mathfrak F}_\ast\gamma_\ast}{4}s\right)ds\right]\nonumber\\
		&=\E\left[t_\ast\right]+\frac{4}{\overline{\mathfrak{F}}_\ast}\E\left[\frac{1}{\gamma_\ast}\right]<+\infty\label{end2}
	\end{align}
	Thus, applying the dominated convergence theorem to the second term on the right-hand-side of \eqref{integral_split} and using \eqref{LLNG1}, we obtain
	\begin{equation}\label{eqq1}
		\int_{0}^{+\infty}\E\left[\exp\left(- \overline{\mathfrak{F}}_\ast\int_{0}^s \gamma(r) dr \right)\right]\overline{\mathfrak{S}}_\ast \overline{\mathfrak{F}}_\ast ds=1.
	\end{equation}  
	Next by a change of variables in \eqref{eqq1} and the fact that $\overline{\mathfrak{S}}_\ast=\frac{1}{R_0},$ we obtain \eqref{ibar1}. 
	To conclude, Lemma~\ref{lemma:H} below implies that the equation $ H(x) = R_0 $ has a unique positive solution if and only if $ R_0 > \E \left[ \frac{1}{\gamma_*} \right] $, which yields the result.  
	On the other hand, as $F^c_0(t)\to0$ as $t\to+\infty$ and $(\overline{\mathfrak S}(t),\overline{\mathfrak F}(t))\to(\overline{\mathfrak S}_\ast,\overline{\mathfrak F}_\ast)$ as $t\to+\infty$, applying the dominated convergence theorem to \eqref{eqn-barI} we obtain
	\begin{equation*}
		\overline{I}_\ast=\lim_{t\to+\infty}\int_{0}^{t}F^c(t-s)\overline{\mathfrak F}(s)\overline{\mathfrak S}(s)ds=\overline{\mathfrak S}_\ast\overline{\mathfrak F}_\ast\E[\eta]=\frac{\overline{\mathfrak F}_\ast}{R_0}\E[\eta].
	\end{equation*}
	This proves the claims of Theorem~\ref{conv1-f} $(i)$.
\end{proof}

Define $H: (0,\infty) \to \mathbb{R}_+ \cup \lbrace \infty \rbrace$ as follows,
\begin{equation} \label{def:H}
	H(x):=\int_{0}^{+\infty}\E\left[\exp\left(-\int_{0}^s \gamma\left(\frac{r}{x}\right) dr \right)\right]ds, \text{ for } x>0.
\end{equation}

\begin{lemma} \label{lemma:H}
	Under Assumptions~\ref{hyp-ga-1} and \ref{hyp-ga}, and if $ \E\left[ \frac{1}{\gamma_*} \right] < + \infty $,
	then $ H(x) < \infty $ for all $ x \in (0,\infty) $ and $ H $ is continuous and strictly increasing on $ (0,\infty) $.
	Moreover,
	\begin{align*}
		\lim_{x \to +\infty} H(x) = +\infty,
	\end{align*}
	and $ H $ can be extended to a continuous function on $ \R_+ $ by setting
	\begin{equation} \label{lim_H0}
		H(0) := \lim_{x \downarrow 0} H(x) = \E\left[ \frac{1}{\gamma_*} \right].
	\end{equation}
\end{lemma}

\begin{proof}
	Let us start by showing that $ H(x) < \infty $ for all $ x \in (0,\infty) $.
	By Assumption~\ref{hyp-ga},
	\begin{equation*}
		\gamma(r) \geq \frac{\gamma_*}{2} \indic{r \geq t_*}.
	\end{equation*}
	Then,
	\begin{align}
		\int_{0}^{\infty} \exp\left( -\int_{0}^{s} \gamma\left( \frac{r}{x} \right) dr \right) ds &\leq x t_* + \int_{x t_*}^{\infty} \exp\left( - \frac{\gamma_*}{2} (s - x t_*) \right) ds \nonumber \\
		&= x t_* + \frac{2}{\gamma_*}. \label{bound_integrand}
	\end{align}
	Note that, in the last step, we use the fact that $ \gamma_* > 0 $ almost surely, since $ \E\left[ \frac{1}{\gamma_*} \right] < \infty $.
	Taking the expectation on both sides of the above inequality and using the fact that $ \E[t_*] < \infty $ and $ \E\left[ \frac{1}{\gamma_*} \right] < \infty $, by Fubini's theorem we obtain that $ H(x) < \infty $ for all $ x \in (0,\infty) $.
	
	Since $ t \mapsto \gamma(t) $ is non-decreasing, we directly obtain that $ x \mapsto H(x) $ is also non-decreasing.
	Moreover, for $ x > 0 $, $H(x)$ is also given by the following equivalent formula
	\begin{equation} \label{alternative_expr_H}
		H(x)=x\int_0^\infty \E\left[\exp\left(-x\int_0^s \gamma(r)dr\right)\right]ds\,.
	\end{equation}
	Using \eqref{bound_integrand} and the dominated convergence theorem, $H$ is therefore continuous on $ (0,\infty) $.
	To prove that $ H $ is strictly increasing, assume by contradiction that $ H(x)=H(y) $ for some $ y > x $. 
	Then, for almost every $ r > 0 $, $ \gamma(r/x)=\gamma(r/y) $ almost surely, which implies that $ \gamma $ is constant, leading to a contradiction.
	
	Since $ \gamma(0) = 0 $, by \eqref{def:H} and the monotone convergence theorem,
	\begin{equation*}
		\lim_{x \to \infty} H(x) = + \infty.
	\end{equation*}
	Similarly, using the fact that $ \gamma $ is non-decreasing, \eqref{bound_integrand} and the dominated convergence theorem, we obtain
	\begin{equation*}
		\lim_{x \downarrow 0} H(x) = \int_{0}^{\infty} \E \left[ \exp\left( - s \gamma_* \right) \right] ds.
	\end{equation*}
	By Fubini's theorem, using the fact that $ \E \left[ \frac{1}{\gamma_*} \right] < \infty $, we obtain \eqref{lim_H0}, which concludes the proof of the lemma.
\end{proof}

\begin{proof}[Proof of Theorem~\ref{conv1-f}~$(ii)$] 
	The goal of Theorem~\ref{conv1-f}~$(ii)$ is to prove that, if $\gamma_\ast$ is deterministic, $R_0>\frac{1}{\gamma_\ast}$, for all $\delta>0$ there exists $t_\delta$ such that $\gamma(t_\delta)\geq(1-\delta)\gamma_\ast$ and if there exists a positive decreasing function $h$ such that for all $0\leq s, t,\,\overline{\lambda}(t+s)\geq h(t)\overline{\lambda}(s),$ and $\overline{\mathfrak F}(0)>0,$ then there exists $c>0$ such that for all $t\geq0,\,\overline{\mathfrak F}(t)\geq c$.
	
	Let $\delta>0$ be such that $(1-\delta)\gamma_\ast>\frac{1}{R_0}$.
	From Assumption~\ref{ASS2}, there exists $s_1\geq0$ deterministic such that $\gamma_{0}(s_1)\wedge\gamma(s_1)\geq(1-\delta)\gamma_\ast$ a.s.
	Let $\epsilon>0$ be such that $(1-\delta)\gamma_\ast>\frac{1+\epsilon}{R_0}$.
	Let $x$ be the solution of the following Volterra equation
	\begin{equation}\label{V1}
		x(t)= h(s_1+t)+(1+\epsilon)\int_{0}^{t}p(t-s)x(s)ds,\;t\geq0, 
	\end{equation}
	with
	\begin{align*}
		p(t)=\frac{\overline{\lambda}(t)}{R_0}.
	\end{align*}
	As 
	\[(1+\epsilon)\int_{0}^{+\infty}p(t)dt=1+\epsilon>1,\qquad\int_{0}^{+\infty}h(t)dt<+\infty\] where the integrability of $h$ results from Assumption~\ref{ASS3} and the integrability of $\overline\lambda$. Moreover, since $h$ and $p$ are bounded and non-negative, by \cite[Remark following Theorem~$4$, page~$253$]{feller2015integral}, $x(t)\to+\infty$ as $t\to+\infty$, hence there exists $s_2\geq0$ such that $x(s_2)>2$.
	
	Let $c_2>0$ be such that 
	\begin{equation}\label{inR_0}
		(1-\delta)\gamma_\ast\exp(-c_2(s_1+s_2))\geq\frac{1+\epsilon}{R_0}.
	\end{equation}
	Let \[c_1=\frac{1}{2}\min(c_2,\overline{\mathfrak F}(0)),\, \quad \text{ and } \quad c_0=\frac{c_1}{2}h(s_1+s_2),\]
	\[t_0=\inf\{t\geq0,\,\overline{\mathfrak F}(t)\leq c_0\},\quad \text{ and } \quad t_1=\sup\{t\leq t_0,\,\overline{\mathfrak F}(t)\geq c_1\}.\]  
	Since $h$ is decreasing and $h(0)=1$ we have $c_0<c_1$.
	We want to show that $t_0=\infty$, and will prove it by contradiction. 
	Let us thus suppose that $t_0<+\infty$, which implies that $t_1<+\infty$. From Assumption~\ref{ASS3}, by the continuity of $\overline{\mathfrak F}$ and the definition of $t_1$, for all $t\geq t_1$, we obtain
	\begin{align*}
		\overline{\mathfrak F}(t)&=\overline{\lambda}_0(t)\overline{I}(0)+\int_{0}^{t}\overline{\lambda}(t-s)\overline{\mathfrak F}(s)\overline{\mathfrak S}(s)ds\\
		&\geq\overline{\lambda}_0(t-t_1+t_1)\overline{I}(0)+\int_{0}^{t_1}\overline{\lambda}(t-t_1+t_1-s)\overline{\mathfrak F}(s)\overline{\mathfrak S}(s)ds\\
		&\geq h(t-t_1)\left(\overline{\lambda}_0(t_1)\overline{I}(0)+\int_{0}^{t_1}\overline{\lambda}(t_1-s)\overline{\mathfrak F}(s)\overline{\mathfrak S}(s)ds\right)\\
		&= h(t-t_1)\overline{\mathfrak F}(t_1)\\
		&\geq c_1 h(t-t_1). 
	\end{align*}
	The definition of $t_0$ and the continuity of $\overline{\mathfrak F},$ implies that $\overline{\mathfrak F}(t_0)\leq c_0.$ Combining with the last inequality evaluated at $t=t_0$, we have $c_0\geq c_1 h(t_0-t_1)$. Hence, by the definition of $c_0$ and the fact that $h$ is decreasing, we deduce that $t_0-t_1> s_1+s_2$. So $t_0>t_1+s_1+s_2$ and for all $t\in[t_1,t_0]$
	\begin{equation}\overline{\mathfrak F}(t)\leq c_1< c_2.\label{e-q}\end{equation}
	
	On the other hand, as $\gamma_0(t)\leq1$ and $\gamma(t)\leq1$,  we have, for all $t\geq t_1$,
	\begin{align*}
		\overline{\mathfrak S}(t)&=\E\left[\gamma_{0}(t)\exp\left(-\int_{0}^{t}\gamma_{0}(r)\overline{\mathfrak F}(r)dr\right)\right] \\
		& \hspace{1.5cm} +\int_{0}^{t}\E\left[\gamma(t-s)\exp\left(-\int_{s}^{t}\gamma(r-s)\overline{\mathfrak F}(r)dr\right)\right]\overline{\mathfrak F}(s)\overline{\mathfrak S}(s)ds\\
		&\geq\exp\left(-\int_{t_1}^{t}\overline{\mathfrak F}(r)dr\right)\left(\E\left[\gamma_{0}(t)\exp\left(-\int_{0}^{t_1}\gamma_{0}(r)\overline{\mathfrak F}(r)dr\right)\right]\right.\\
		&\hspace{1.5cm}\left.+\int_{0}^{t_1}\E\left[\gamma(t-s)\exp\left(-\int_{s}^{t_1}\gamma(r-s)\overline{\mathfrak F}(r)dr\right)\right]\overline{\mathfrak F}(s)\overline{\mathfrak S}(s)ds\right).
	\end{align*} 
	But, as for $t\in[t_1+s_1,t_1+s_1+s_2]$ and $s\in[0,t_1],\,\gamma_{0}(t)\wedge\gamma(t-s)\geq\gamma_{0}(s_1)\wedge\gamma(s_1)\geq(1-\delta)\gamma_\ast,$ and using  \eqref{eqG1} at time $t_1$ we deduce that, for all $t\in[t_1+s_1,t_1+s_1+s_2]$,
	
	\begin{align*}
		\overline{\mathfrak{S}}(t)&\geq(1-\delta)\gamma_\ast\exp\left(-\int_{t_1}^{t}\overline{\mathfrak F}(r)dr\right).
	\end{align*} 
	Moreover, since from \eqref{e-q} for $t\in[t_1+s_1,t_1+s_1+s_2],\,\overline{\mathfrak F}(t)\leq c_2$,
	\begin{align*}
		\overline{\mathfrak{S}}(t)&\geq(1-\delta)\gamma_\ast\exp(-c_2(s_2+s_1)).
	\end{align*} 
	
	Then from \eqref{inR_0}
	\begin{equation}
		\forall t\in[t_1+s_1,t_1+s_1+s_2],\; \quad \overline{\mathfrak{S}}(t)\geq\frac{1+\epsilon}{R_0}.\label{e-q1}
	\end{equation}
	Let $y(t)=\overline{\mathfrak F}(t+t_1+s_1)$ and define $g$ as follows:
	\begin{equation*}
		g(t)=\overline{I}(0)\overline{\lambda}_0(t_1+s_1+t)+\int_{0}^{t_1+s_1}\overline{\lambda}(t_1+s_1+t-s)\overline{\mathfrak F}(s)\overline{\mathfrak S}(s)ds,
	\end{equation*}
	where we recall that
	\begin{align*}
		p(t)=\frac{\overline{\lambda}(t)}{R_0}.
	\end{align*}
	Then using \eqref{e-q1} for any $t\geq0$,
	\begin{align*}
		y(t)&\geq g(t)+(1+\epsilon)\int_{t_1+s_1}^{t_1+s_1+t}p(t_1+s_1+t-s)y(s-t_1-s_1)ds\\
		&= g(t)+(1+\epsilon)\int_{0}^{t}p(t-s)y(s)ds.
	\end{align*}
	However, from Assumption~\ref{ASS3} we deduce that
	\begin{equation*}
		g(t)\geq \overline{\mathfrak F}(t_1) h(s_1+t)
	\end{equation*}
	and as $\overline{\mathfrak F}(t_1)= c_1$ by continuity, we deduce that 
	\begin{equation}\label{eq-last}
		y(t)\geq c_1 h(s_1+t)+(1+\epsilon)\int_{0}^{t}p(t-s)y(s)ds.
	\end{equation}
	Thus, applying Lemma~9.8.2 in \cite{gripenberg_volterra_1990} to $ -y(t) $ in \eqref{eq-last} with the convolution kernel $ k(s) = (1+\varepsilon)p(s) $, we obtain
	\begin{equation*}
		y(t)\geq c_1 x(t)
	\end{equation*}
	where $x$ is given by \eqref{V1}.
	However $x(s_2)>2$. Hence $\overline{\mathfrak F}(t_1+s_1+s_2)>2c_1>c_1$ and $t_0\geq t_1+s_1+s_2$, this contradicts the definition of $t_1$. Hence $t_0=+\infty$. This concludes the proof. 
\end{proof} 

\bibliographystyle{abbrv}
\bibliography{Epidemic-Age,VIVS_raphael}

\begin{thebibliography}{10}

\bibitem{BarRos}
M.~V. Barbarossa and G.~R{\"o}st.
\newblock Immuno-epidemiology of a population structured by immune status: a
  mathematical study of waning immunity and immune system boosting.
\newblock {\em Journal of Mathematical Biology}, 71:1737--1770, 2015.

\bibitem{billingsley1999convergence}
P.~Billingsley.
\newblock {\em Convergence of probability measures}.
\newblock John Wiley \& Sons, 1999.

\bibitem{breda2012formulation}
D.~Breda, O.~Diekmann, W.~De~Graaf, A.~Pugliese, and R.~Vermiglio.
\newblock On the formulation of epidemic models (an appraisal of kermack and
  mckendrick).
\newblock {\em Journal of Biological Dynamics}, 6(sup2):103--117, 2012.

\bibitem{britton2018stochastic}
T.~Britton and E.~Pardoux.
\newblock {\em Stochastic epidemic models with inference}.
\newblock Springer, 2019.

\bibitem{brunner_volterra_2017}
H.~Brunner.
\newblock {\em Volterra {{Integral Equations}}: {{An Introduction}} to
  {{Theory}} and {{Applications}}}.
\newblock {Cambridge University Press}, Jan. 2017.

\bibitem{Carls}
R.-M. Carlsson, L.~M. Childs, Z.~Feng, J.~W. Glasser, J.~M. Heffernan, J.~Li,
  and G.~R{\"o}st.
\newblock Modeling the waning and boosting of immunity from infection or
  vaccination.
\newblock {\em Journal of Theoretical Biology}, 497:110265, 2020.

\bibitem{chevallier_fluctuations_2017}
J.~Chevallier.
\newblock Fluctuations for mean-field interacting age-dependent {{Hawkes}}
  processes.
\newblock {\em Electronic Journal of Probability}, 22, 2017.

\bibitem{chevallier2017mean}
J.~Chevallier.
\newblock Mean-field limit of generalized {H}awkes processes.
\newblock {\em Stochastic Processes and their Applications},
  127(12):3870--3912, 2017.

\bibitem{fan2020limit}
J.~Y. Fan, K.~Hamza, P.~Jagers, and F.~C. Klebaner.
\newblock Limit theorems for multi-type general branching processes with
  population dependence.
\newblock {\em Advances in Applied Probability}, 52(4):1127--1163, 2020.

\bibitem{feller2015integral}
W.~Feller.
\newblock On the integral equation of renewal theory.
\newblock In {\em Selected Papers I}, pages 567--591. Springer, 2015.

\bibitem{forien2021epidemic}
R.~Forien, G.~Pang, and {\'E}.~Pardoux.
\newblock Epidemic models with varying infectivity.
\newblock {\em SIAM Journal on Applied Mathematics}, 81(5):1893--1930, 2021.

\bibitem{forien2021estimating}
R.~Forien, G.~Pang, and {\'E}.~Pardoux.
\newblock Estimating the state of the {COVID}-19 epidemic in {F}rance using a
  model with memory.
\newblock {\em Royal Society Open Science}, 8(3):202327, 2021.

\bibitem{FPP-PUQR}
R.~Forien, G.~Pang, and {\'E}.~Pardoux.
\newblock Multi-patch multi-group epidemic model with varying infectivity.
\newblock {\em Probab. Uncertain. Quant. Risk}, 7:333--364, 2022.

\bibitem{foutel2025optimal}
F.~Foutel-Rodier, A.~Charpentier, and H.~Gu{\'e}rin.
\newblock Optimal vaccination policy to prevent endemicity: A stochastic model.
\newblock {\em Journal of Mathematical Biology}, 90(1):1--55, 2025.

\bibitem{gripenberg_volterra_1990}
G.~Gripenberg, S.-O. Londen, and O.~Staffans.
\newblock {\em Volterra Integral and Functional Equations}.
\newblock Number~34. Cambridge University Press, 1990.

\bibitem{hamza2013age}
K.~Hamza, P.~Jagers, and F.~C. Klebaner.
\newblock The age structure of population-dependent general branching processes
  in environments with a high carrying capacity.
\newblock {\em Proceedings of the Steklov Institute of Mathematics},
  282(1):90--105, 2013.

\bibitem{hamza2016establishment}
K.~Hamza, P.~Jagers, and F.~C. Klebaner.
\newblock On the establishment, persistence, and inevitable extinction of
  populations.
\newblock {\em Journal of Mathematical Biology}, 72(4):797--820, 2016.

\bibitem{inaba2001kermack}
H.~Inaba.
\newblock Kermack and {M}c{K}endrick revisited: the variable susceptibility
  model for infectious diseases.
\newblock {\em Japan Journal of Industrial and Applied Mathematics},
  18(2):273--292, 2001.

\bibitem{inaba_endemic_2016}
H.~Inaba.
\newblock Endemic threshold analysis for the
  {{Kermack}}\textendash{{McKendrick}} reinfection model.
\newblock {\em Josai Math. Monogr}, 9:105--133, 2016.

\bibitem{inaba_variable_2017}
H.~Inaba.
\newblock Variable {{Susceptibility}}, {{Reinfection}}, and {{Immunity}}.
\newblock In {\em Age-{{Structured Population Dynamics}} in {{Demography}} and
  {{Epidemiology}}}, pages 379--442. {Springer}, 2017.

\bibitem{inaba2004mathematical}
H.~Inaba and H.~Sekine.
\newblock A mathematical model for {C}hagas disease with
  infection-age-dependent infectivity.
\newblock {\em Mathematical Biosciences}, 190(1):39--69, 2004.

\bibitem{jagers2000population}
P.~Jagers and F.~C. Klebaner.
\newblock Population-size-dependent and age-dependent branching processes.
\newblock {\em Stochastic Processes and their Applications}, 87(2):235--254,
  2000.

\bibitem{KMK}
W.~O. Kermack and A.~G. McKendrick.
\newblock A contribution to the mathematical theory of epidemics.
\newblock {\em Proceedings of the Royal Society of London. Series A, Containing
  papers of a mathematical and physical character}, 115(772):700--721, 1927.

\bibitem{kermack_contributions_1932}
W.~O. Kermack and A.~G. McKendrick.
\newblock Contributions to the mathematical theory of epidemics.
  {{II}}.\textemdash{{The}} problem of endemicity.
\newblock {\em Proceedings of the Royal Society of London. Series A, containing
  papers of a mathematical and physical character}, 138(834):55--83, 1932.

\bibitem{kermack_contributions_1933}
W.~O. Kermack and A.~G. McKendrick.
\newblock Contributions to the mathematical theory of
  epidemics\textendash{{III}}. {{Further}} studies of the problem of
  endemicity. 1933.
\newblock {\em Proceedings of the Royal Society of London. Series A, containing
  papers of a mathematical and physical character}, 141(843):89--118, 1933.

\bibitem{Khal-Brit}
M.~E. Khalifi and T.~Britton.
\newblock Extending sirs epidemics to allow for gradual waning of immunity.
\newblock {\em Journal of Royal Society Interface}, 20:20230042, 2023.

\bibitem{oelschlager1990limit}
K.~Oelschlager.
\newblock Limit theorems for age-structured populations.
\newblock {\em Annals of Probability}, 18(1):290--318, 1990.

\bibitem{PandPardoux-2020}
G.~Pang and {\'E}.~Pardoux.
\newblock Functional limit theorems for non-{M}arkovian epidemic models.
\newblock {\em Annals of Applied Probability}, 32(3):1615--1665, 2022.

\bibitem{PP2023}
G.~Pang and {\'E}.~Pardoux.
\newblock {F}unctional law of large numbers and {PDE}s for epidemic models with
  infection-age dependent infectivity.
\newblock {\em Applied Math. \& Optimization}, 87, 2023.

\bibitem{rao1963law}
R.~R. Rao.
\newblock The law of large numbers for {$D[0,1]$}-valued random variables.
\newblock {\em Theory of Probability \& Its Applications}, 8(1):70--74, 1963.

\bibitem{singhal2020review}
T.~Singhal.
\newblock A review of coronavirus disease-2019 (covid-19).
\newblock {\em The indian journal of pediatrics}, 87(4):281--286, 2020.

\bibitem{sofonea2021memory}
M.~T. Sofonea, B.~Reyn{\'e}, B.~Elie, R.~Djidjou-Demasse, C.~Selinger,
  Y.~Michalakis, and S.~Alizon.
\newblock Memory is key in capturing {COVID}-19 epidemiological dynamics.
\newblock {\em Epidemics}, 35:100459, 2021.

\bibitem{sznitman1991topics}
A.-S. Sznitman.
\newblock Topics in propagation of chaos.
\newblock In {\em Ecole d'{\'e}t{\'e} de probabilit{\'e}s de Saint-Flour
  XIX—1989}, pages 165--251. Springer, 1991.

\bibitem{ThieYang}
H.~R. Thieme and J.~Yang.
\newblock An endemic model with variable re-infection rate and applications to
  influenza.
\newblock {\em Mathematical Biosciences}, 180(1-2):207--235, 2002.

\bibitem{tran_large_2008}
V.~C. Tran.
\newblock Large population limit and time behaviour of a stochastic particle
  model describing an age-structured population.
\newblock {\em ESAIM: Probability and Statistics}, 12:345--386, 2008.

\bibitem{tuteja2007malaria}
R.~Tuteja.
\newblock Malaria- an overview.
\newblock {\em The FEBS journal}, 274(18):4670--4679, 2007.

\bibitem{webb1985theory}
G.~F. Webb.
\newblock {\em Theory of nonlinear age-dependent population dynamics}.
\newblock CRC Press, 1985.

\bibitem{ngoufack2024functional}
A.~B. Zotsa~Ngoufack.
\newblock Functional central limit theorems for epidemic models with varying
  infectivity and waning immunity.
\newblock {\em ESAIM: PS}, 29:45--112, 2025.

\end{thebibliography}

\end{document}